\documentclass[a4paper,reqno,12pt]{amsart}
\usepackage{enumitem}
\usepackage{amsmath,amsfonts,amssymb}
\usepackage{graphics,color}
\usepackage[latin1]{inputenc}
\usepackage[colorlinks = true,
            linkcolor = blue,
            urlcolor  = blue,
            citecolor = blue,
            anchorcolor = blue]{hyperref}

\newtheorem{theorem}{Theorem}[section]
\newtheorem{lemma}[theorem]{Lemma}

\makeatletter
\@namedef{subjclassname@2020}{\textup{2020} Mathematics Subject Classification}
\makeatother

\numberwithin{equation}{section}
\newtheorem{proposition}[theorem]{Proposition}
\newtheorem{corollary}[theorem]{Corollary}

\newtheorem{remark}[theorem]{Remark}
\usepackage{cite}
\numberwithin{equation}{section}

\allowdisplaybreaks[4]
\begin{document}
\title[Defect thresholds and concentration]{Sharp Defect Thresholds and Concentration for Radial Elliptic Equations and Systems}

\author{Shengbing Deng}
\address{School of Mathematics and Statistics, Southwest University, Chongqing 400715, People's Republic of China.}
\email{shbdeng@swu.edu.cn}

\author{Bo Zhang$^*$}
\address{School of Mathematics and Statistics, Southwest University, Chongqing 400715, People's Republic of China.}
\email{zhangbo371013@163.com}

\footnotetext[1]{Corresponding author.}



\begin{abstract}
In this paper, we consider positive radial entire solutions of nonlinear elliptic equations with bounded positive coefficients and optimize the negative logarithmic derivative of the coefficient. For the scalar equation, we determine the exact, unattained defect threshold and establish quantitative gaps when the coefficient ratio is fixed. Near the threshold, we prove a sharp stability estimate and describe the loss of compactness through a profile decomposition in logarithmic radius, including a characterization of the possible diffuse energy and conditions that exclude it. We also show that the resulting profile configurations can be realized by suitable monotone coefficients. For unequal weighted elliptic systems, we obtain the corresponding sharp defect threshold and fixed-ratio gap, and analyze near-threshold sequences through local homogeneous limits and logarithmic spreading. A BV relaxation and smoothing argument further relates the relaxed and smooth optimization problems.

\end{abstract}

\maketitle
\enlargethispage{2pt}

\section{Introduction and main results}

The existence of positive entire solutions of
\[
-\Delta U=c(|x|)U^p\quad\hbox{in }\mathbb R^n,\qquad n\ge3,
\]
depends on the interaction between the exponent and the radial variation
of the coefficient. For constant positive coefficients, the subcritical
Liouville theorem of Gidas and Spruck~\cite{GS} excludes positive entire
solutions. For prescribed radial coefficients, generalized Pohozaev
identities and radial structure theorems give existence, nonexistence
and decay criteria~\cite{KNY,KYY,KYYm,YY,Yanagida}. Allowing the coefficient
to be chosen with the solution changes this picture: Lei and
Li~\cite{LeiLi} give existence criteria for some double-bounded
coefficients, including scalar existence when $p>n/(n-2)$.
The question considered here is quantitative. With $p$ fixed, how much
radial decrease must a bounded positive coefficient permit before a
positive radial entire solution can exist?

The absolute size of the coefficient does not measure this obstruction.
For every $\lambda>0$, replacing $(c,U)$ by
$(\lambda c,\lambda^{-1/(p-1)}U)$ preserves the equation.
A quantity adapted to its scaling is instead the relative variation
\[
-\frac{rc'(r)}{c(r)}=-\frac{d}{d\log r}\log c(r).
\]
It measures decrease per unit logarithmic radius. We therefore set
\begin{equation}\label{eq:defectdefinition}
D(c)=\sup_{r>0}\left[-\frac{rc'(r)}{c(r)}\right]_+,
\qquad x_+=\max\{x,0\}.
\end{equation}
This one-sided, dimensionless defect is unchanged by multiplication of
$c(r)$ by a positive constant or by dilation of its radial argument.

Throughout the scalar discussion we assume
\begin{equation}\label{eq:radialrange}
n\ge3,\qquad \frac n{n-2}<p<\frac{n+2}{n-2},\qquad N=n-2.
\end{equation}
Let $\mathcal A_{n,p}$ consist of positive radial pairs $(c,U)$ with
$c\in C^1(\mathbb R^n)$, $U\in C^2(\mathbb R^n)$ and $-\Delta U=cU^p$,
such that
\[
0<\inf c\le\sup c<\infty,\qquad
\sup_{r>0}|rc'(r)/c(r)|<\infty.
\]
These bounds belong to each pair and need not be uniform along a sequence.
The first optimization problem is
\begin{equation}\label{eq:radialinfimum}
D_{\mathrm{crit}}(n,p)=\inf_{(c,U)\in\mathcal A_{n,p}}D(c).
\end{equation}
The critical value is selected by the Pohozaev balance. Indeed, the radial
Pohozaev function defined in \eqref{eq:radialpohozaev} satisfies
\[
P'(r)=\frac{r^{n-1}c(r)U(r)^{p+1}}{p+1}
\left(d_*+\frac{rc'(r)}{c(r)}\right),
\qquad d_*=2-\frac{N(p-1)}2.
\]
Its endpoint values vanish. If $D(c)\le d_*$, it is nondecreasing and,
by regularity at zero, strictly increasing near zero, a contradiction.
Thus the classical obstruction gives $D(c)>d_*$; it remains to determine
whether it can be approached within the double-bounded coefficient class.

The exclusion of equality also raises the question of what must
degenerate if this obstruction is sharp. To separate a small local
defect excess from large global variation, introduce the coefficient ratio
\begin{equation}\label{eq:contrastinfimum}
\mathcal R(c)=\frac{\sup c}{\inf c},
\qquad
D_{\mathrm{crit}}(n,p;M)=
\inf_{\substack{(c,U)\in\mathcal A_{n,p}\\\mathcal R(c)\le M}}D(c).
\end{equation}
We use $\inf\varnothing=+\infty$. A fixed ratio bounds the global contrast
without choosing an absolute coefficient scale.

\begin{theorem}\label{thm:radialthreshold}\label{thm:contrast}
Under \eqref{eq:radialrange},
\begin{equation}\label{eq:radialsharp}
D_{\mathrm{crit}}(n,p)=d_*,
\end{equation}
and the infimum is not attained. For each finite $M\ge1$ there is
$g_{n,p}(M)>0$ such that every $(c,U)\in\mathcal A_{n,p}$ with
$\mathcal R(c)\le M$ satisfies $D(c)\ge d_*+g_{n,p}(M)$.
Consequently, $D(c_j)\to d_*$ forces $\mathcal R(c_j)\to\infty$.
\end{theorem}

The sharpness construction produces smooth nonincreasing coefficients
and positive fast-decay solutions. It also yields quantitative upper
bounds for the fixed-ratio optimization, while the Pohozaev identity gives
quantitative lower bounds; their precise forms are established in
Section~\ref{sec:radial}. The lower obstruction is classical in origin.
The construction and the fixed-ratio gap identify how it is approached
when both the coefficient and the solution are allowed to vary.

Coefficient-ratio degeneration does not by itself describe the solutions.
To study their compactness independently of amplitude and radial scale,
we normalize at a maximum of $r^NU(r)^2$ and use the invariant energy
\[
H=\frac{\displaystyle\int_0^\infty r^{n-1}|U'|^2\,dr}
        {\displaystyle\max_{r>0}r^NU(r)^2}.
\]
If $r_0$ is a maximizing radius, set $V(r)=U(r_0r)/U(r_0)$.
Then $V(1)=1$, $r^NV(r)^2\le1$, and its Dirichlet energy is $H$.
The profile selected by zero Pohozaev balance is
\[
V_*(r)=\left(\frac2{1+r^a}\right)^{N/a},\qquad
a=2-d_*,\qquad H_*=\int_0^\infty r^{n-1}|V_*'|^2\,dr.
\]
It solves the equation with coefficient
$c_*(r)=N(N+a)r^{-d_*}/4$. This critical pair lies outside
$\mathcal A_{n,p}$: the coefficient is singular at zero and loses its
positive lower bound at infinity.

The scale symmetry becomes especially transparent in logarithmic radius.
Writing
\[
\tau=\log r,\qquad W(\tau)=e^{N\tau/2}V(e^\tau),\qquad
W_*(\tau)=\cosh(a\tau/2)^{-N/a},
\]
turns the energy-preserving radial dilation into translation in $\tau$.
Near the least limiting energy one expects a single profile; a general
energy bound allows further profiles at separated logarithmic scales.
The relevant background is concentration-compactness~\cite{Lions1984},
global compactness~\cite{Struwe1984} and profile extraction~\cite{Solimini1995}.
Here the small Pohozaev error identifies each nonzero translated limit,
but a part of the energy may remain outside all such profiles.

\begin{theorem}\label{thm:scalaroverview}
For $(c,U)\in\mathcal A_{n,p}$, let $V$ be any peak normalization and set
$\varepsilon=D(c)-d_*$. There is $C=C(n,p)>0$ such that
\[
\left|\int_0^\infty r^{n-1}|V'-V_*'|^2\,dr-(H-H_*)\right|
\le C\min\{\varepsilon H,1\}.
\]
If $D(c_j)\to d_*$ and $H_j\to h<\infty$, then, after passage to a
subsequence, the nonzero concentrations of $W_j$ are represented by
translated profiles $A_\ell W_*$, with $0<A_\ell\le1$ and $A_1=1$.
Their centers separate,
their energies decouple, and
\[
h=H_*\sum_\ell A_\ell^2+h_{\mathrm{diff}},\qquad h_{\mathrm{diff}}\ge0.
\]
There may be finitely or countably many profiles, and $h_{\mathrm{diff}}$
is the remaining energy. In particular, if $H_j\to H_*$,
then $V_j\to V_*$ strongly in the radial Dirichlet norm.
\end{theorem}

The precise statements are Theorems~\ref{thm:scalarminimum}
and~\ref{thm:scalarprofiles}. The value $H_*$ is the minimum possible
finite limiting energy at the threshold, rather than a lower bound for
every admissible pair. Section~\ref{sec:scalarcompactness} also determines
the optimal linear order of the stability error and of the minimum-energy
correction in a small defect window. The critical profile belongs to the
Caffarelli--Kohn--Nirenberg family. In contrast to the functional deficit
and Euler--Lagrange residual used by Wei and Wu~\cite{WeiWu2022}, the
control quantity here is $(D(c)-d_*)H$, obtained directly from the
coefficient variation through the Pohozaev identity.

The diffuse term records energy which no nonzero translated profile
captures. Its location is detected by the coefficient after the critical
power $r^{-d_*}$ has been factored out. Each profile sees a finite constant
value of this transformed coefficient, whereas the diffuse nonlinear
energy escapes to arbitrarily large coefficient values.
This gives an exact criterion for the disappearance of the diffuse
remainder. Uniform positive moments of the transformed coefficient with
respect to nonlinear energy imply this tightness; a critical moment
additionally bounds the number of profiles. These conclusions concern
strong decomposition after subtraction of all profiles. They do not,
without a further restriction, give convergence to a single profile.

The decomposition also has a converse at the level of amplitudes and
remaining energy. Every finite or countable square-summable list of
positive amplitudes with maximum one, together with any prescribed
nonnegative diffuse energy, is realized by a near-threshold sequence
with nonincreasing coefficients and positive fast tails
(Theorem~\ref{thm:energyrealization}). The coefficients vary with the
sequence; no fixed-coefficient realization is asserted.
Bubble towers have substantial precedents in Matukuma-type and scalar
curvature equations~\cite{AlarconQuaas2010,DalbonoFrancaSfecci2022}.
In particular, the weighted extension in
\cite[Corollary~1]{DalbonoFrancaSfecci2022} includes a full coefficient
$r^{-d_*}K(r)$ with double-bounded $K$. The present construction keeps
the full coefficient double bounded, approaches the defect threshold,
and prescribes the profile amplitudes and diffuse energy.
Period insertion also makes $H$ unbounded in every positive defect
window, so an energy bound cannot be inferred from near-criticality alone.

The scalar picture leads naturally to a coupled problem. If the two
equations have different weights or nonlinear powers, their scaling
exponents need not agree, and a critical balance must involve both
coefficient defects. We consider
\[
\begin{cases}
-\Delta u=|x|^\alpha a_0(|x|)v^p,\\
-\Delta v=|x|^\beta b_0(|x|)u^q.
\end{cases}
\]
Rellich--Pohozaev identities and the radial ground-state theory identify
the critical hyperbola~\cite{Mitidieri,Phan,BG}; the asymptotics
in~\cite{BG} supply the critical orbits used below.
The integral formulation~\cite{ChenLi} and existence criteria for some
double-bounded coefficient factors~\cite{LeiLi,Villavert} provide further
background. Our assumptions are
\begin{equation}\label{eq:systemradialrange}
\begin{gathered}
n\ge3,\quad p,q>0,\quad pq>1,\quad \alpha,\beta>-2,\\
N=n-2,\quad A=2+\alpha,\quad B=2+\beta,\\
s=\frac{A+pB}{pq-1}<N,\qquad
t=\frac{B+qA}{pq-1}<N,\\
\gamma_0=\frac{n+\alpha}{p+1}+\frac{n+\beta}{q+1}-N>0.
\end{gathered}
\end{equation}
The positive exponents $s,t$ describe the component scaling for constant
coefficient factors; their restrictions are explained in
Proposition~\ref{prop:existencerange}.
Let $\mathcal B=\mathcal B_{n,p,q}^{\alpha,\beta}$ be the class of positive
radial quadruples solving this system for $r>0$, with
$u,v\in C(\mathbb R^n)\cap C^2(\mathbb R^n\setminus\{0\})$ and
$a_0,b_0\in C^1((0,\infty))$. Both coefficient factors have positive
upper and lower bounds and bounded absolute logarithmic derivatives,
with bounds allowed to depend on the quadruple.

The Pohozaev identity selects the combined defect
\[
\mathcal D(a_0,b_0)=\frac{D(a_0)}{p+1}+\frac{D(b_0)}{q+1},\qquad
\mathcal R(a_0,b_0)=\max\{\mathcal R(a_0),\mathcal R(b_0)\}.
\]
Write $\mathcal D^{sys}_{\mathrm{crit}}$ for its infimum over $\mathcal B$ and
$\mathcal D^{sys}_{\mathrm{crit}}(M)$ for the infimum with
$\mathcal R(a_0,b_0)\le M$. On the scalar diagonal $p=q$,
$\alpha=\beta=0$ and $a_0=b_0=c$,
$\mathcal D(c,c)=2D(c)/(p+1)$ and $\gamma_0=2d_* /(p+1)$.
For Newton kernels and power responses, the strict criterion of
Hu, Long and Zou~\cite{HLZ2026} corresponds to
$\mathcal D(a_0,b_0)<\gamma_0$. The radial endpoint argument also
excludes equality, and the coupled construction proves sharpness.

\begin{theorem}\label{thm:systemthreshold}\label{thm:systemcontrast}
Under \eqref{eq:systemradialrange},
\begin{equation}\label{eq:systemstrict}
\mathcal D(a_0,b_0)>\gamma_0\quad\hbox{for every quadruple in }\mathcal B,
\end{equation}
whereas $\mathcal D^{sys}_{\mathrm{crit}}=\gamma_0$.
For each finite $M\ge1$ there is $g(M)>0$, depending only on $M$ and
the fixed equation parameters, such that
$\mathcal D(a_0,b_0)\ge\gamma_0+g(M)$ whenever
$\mathcal R(a_0,b_0)\le M$. Hence every near-threshold sequence has
$\max\{\mathcal R(a_{0,j}),\mathcal R(b_{0,j})\}\to\infty$.
\end{theorem}

Section~\ref{sec:systems} gives quantitative gap estimates and constructs
approximating solutions with constant coefficient factors and exact power
tails for large radius. Only the maximum of the two ratios is forced to
diverge. At fixed ratios, finite infima are attained after BV relaxation;
the relaxed value is the right limit of the smooth optimization value.
This does not assert attainment in the original smooth class.

For the corresponding compactness problem, the invariant normalized
energy is
\[
H=\frac{\displaystyle\int_0^\infty r^{n-1}u'v'\,dr}
        {\displaystyle\max_{r>0}r^Nu(r)v(r)}.
\]
Both quantities are finite and positive. Peak normalization means dilating
a maximizing radius to one and dividing each component by its value
there. A small combined defect excess controls the normalized Pohozaev
error through its product with $H$, leading to the following alternatives.

\begin{theorem}\label{thm:systemoverview}
Let a sequence in $\mathcal B$ satisfy
$\delta_j=\mathcal D(a_{0,j},b_{0,j})-\gamma_0\to0$.
After passage to a subsequence, either $H_j$ is bounded or $H_j\to\infty$.
In the bounded case, any peak normalization has a further subsequence
converging in $C^2_{\mathrm{loc}}((0,\infty))$ to a positive solution of
a critical system with positive homogeneous coefficient factors and finite
nonlinear energies. No regularity of this limit at zero is asserted.
In the unbounded case, the fraction of nonlinear energy in every interval
of fixed logarithmic length tends to zero, uniformly over the centers of
these intervals.
\end{theorem}

Theorem~\ref{thm:sequencelimits} gives the complete local limit statement.
If $\delta_jH_j\to0$, the homogeneous coefficient limits persist even
when $H_j\to\infty$; local convergence to a nonzero profile can therefore
coexist with logarithmic spreading of the total energy.
The scalar periodic construction realizes this behavior in the diagonal
subclass. For unequal systems, however, the cross energy is not a positive
definite squared norm on arbitrary variations, and the scalar energy
splitting argument does not directly extend. A complete global profile
decomposition for unequal systems is not claimed here.

The results thus address three aspects of the same optimization problem:
the sharp threshold and its fixed-ratio gap; scalar stability, profile
decomposition and realization near that threshold; and the threshold
and local concentration mechanisms for unequal systems.
The two main questions left open are the optimal asymptotics of the
fixed-ratio gap and a full profile decomposition in the unequal case.
Section~\ref{sec:radial} treats the scalar threshold and constructions,
and Section~\ref{sec:scalarcompactness} studies the scalar near-threshold
sequences. Sections~\ref{sec:systems} and~\ref{sec:limits} establish the
system threshold, relaxation and limit results. Appendix~\ref{sec:newtonbackground}
records the Newton representation, admissible exponent range and an
explicit example.

\section{Sharp thresholds for scalar Newton equations}\label{sec:radial}

We first prove the strict Pohozaev obstruction. Its sharpness then
requires a simultaneous construction of the coefficient and solution,
with admissible behavior at both endpoints. The last subsection
quantifies the cost of keeping the coefficient ratio bounded and
identifies the critical profile of the constructed family.

Throughout the paper, primes denote differentiation with respect to the radial
variable $r$ when applied to radial functions. After the logarithmic change
of variables $\tau=\log r$, dots denote differentiation with respect to $\tau$.
Thus, for a function $f=f(\tau)$, we write
\[
\dot{f}=\frac{df}{d\tau},\qquad
\ddot{f}=\frac{d^2f}{d\tau^2}.
\]
In particular, for a radial function $g=g(r)$,
\[
\frac{d}{d\tau}g(e^\tau)=e^\tau g'(e^\tau).
\]
Primes on functions explicitly defined in an auxiliary variable, such as
$s$, denote differentiation with respect to that variable. Derivatives with
respect to auxiliary variables are indicated explicitly whenever ambiguity
may arise.

Throughout this section, $n,p$ satisfy \eqref{eq:radialrange}. Write
\[
N=n-2,\qquad \mu=\frac2{p-1},\qquad
d_*=2-\frac{N(p-1)}2,\qquad
\lambda=N(p-1)-2.
\]
Then
\[
\frac N2<\mu<N,\qquad 0<d_*<1,\qquad 0<\lambda<2.
\]
The defect is defined in \eqref{eq:defectdefinition}. The coefficient bounds may depend on $(c,U)$. The scalar exponent $\mu$ equals both scaling exponents when $p=q$ and $\alpha=\beta=0$.

\subsection{The exact defect threshold}

\begin{proposition}\label{prop:radialendpoint}
Every $(c,U)\in\mathcal A_{n,p}$ satisfies $D(c)>d_*$.
\end{proposition}
\begin{proof}
The radial equation and regularity at zero give
\[
-r^{n-1}U'(r)=\int_0^r s^{n-1}c(s)U(s)^p\,ds>0.
\]
Thus $U$ is decreasing. The lower coefficient bound gives
\[
-U'(r)\ge\frac{c_-}{n}rU(r)^p,\qquad
\frac{d}{dr}U(r)^{1-p}\ge\frac{(p-1)c_-}{n}r.
\]
Integration from zero to $r$ yields
\[
U(r)\le
\left(U(0)^{1-p}+\frac{(p-1)c_-}{2n}r^2\right)^{-1/(p-1)}.
\]
Since $\mu p=\mu+2$ and $n-\mu p=N-\mu>0$, the upper coefficient bound and the flux formula give, for $r\ge1$,
\[
|U'(r)|\le Cr^{1-n}
\left(1+\int_1^r s^{n-1-\mu p}\,ds\right)
\le Cr^{-\mu-1}.
\]
Consequently
\begin{equation}\label{eq:radialdecay}
U(r)\le Cr^{-\mu},\qquad
|U'(r)|\le Cr^{-\mu-1}\qquad(r\ge1).
\end{equation}
The constant in \eqref{eq:radialdecay} depends only on $n,p,c_-,c_+$ and $U(0)$. No derivative bound on $c$ was used.

Define the radial Pohozaev quantity
\begin{equation}\label{eq:radialpohozaev}
P(r)=r^n\left(\frac{U'(r)^2}{2}
+\frac{c(r)U(r)^{p+1}}{p+1}\right)
+\frac N2r^{n-1}U(r)U'(r).
\end{equation}
Differentiate \eqref{eq:radialpohozaev} and use $U''=-(n-1)U'/r-cU^p$. The $cU^pU'$ terms cancel. Since
\[
\frac n2-(n-1)+\frac N2=0,\qquad
\frac n{p+1}-\frac N2=\frac{d_*}{p+1},
\]
we obtain
\begin{equation}\label{eq:radialpohozaevderivative}
P'(r)=\frac{r^{n-1}c(r)U(r)^{p+1}}{p+1}
\left(d_*+\frac{rc'(r)}{c(r)}\right).
\end{equation}
At zero, $U'(r)=O(r)$ and $c,U$ are bounded. Hence $P(r)\to0$. At infinity, \eqref{eq:radialdecay} gives
\[
|P(r)|\le Cr^{N-2\mu}\longrightarrow0,
\]
since $N-2\mu<0$.

Suppose $D(c)\le d_*$. Formula \eqref{eq:radialpohozaevderivative} gives $P'\ge0$. Moreover, $rc'(r)/c(r)\to0$ as $r\downarrow0$, because $c$ is $C^1$ and $c(0)>0$. Since $d_*>0$, $P'>0$ for all sufficiently small positive $r$. This contradicts the two zero limits of $P$. Therefore $D(c)>d_*$.
\end{proof}

\subsection{Near-critical constructions}\label{subsec:scalarconstruction}

The strict lower bound does not yet prove sharpness. We must approach
the critical logarithmic slope while restoring positive bounded
coefficients at both zero and infinity. The construction has the
following precise properties.

\begin{proposition}\label{prop:scalarconstruction}
Assume \eqref{eq:radialrange}. For every $\varepsilon>0$ there is a pair
$(c,U)\in\mathcal A_{n,p}$ with $c,U\in C^\infty(\mathbb R^n)$ such that
\[
\begin{gathered}
c(0)=1,\qquad c'(r)\le0,\qquad c(\infty)>0,\\
d_*<D(c)\le d_*+\varepsilon,\qquad
0<\lim_{r\to\infty}r^{n-2}U(r)<\infty.
\end{gathered}
\]
\end{proposition}

Logarithmic radius turns the scaling structure into a phase-space
construction. Rather than prescribe $c$ first, we choose a solution
slope and recover the coefficient from the equation. The critical
homogeneous portions correspond to autonomous dynamics; transitions
connect them to the bounded-coefficient tails.
Put $\tau=\log r$. For a decreasing positive radial solution, set
\[
y=-\frac{e^\tau U'(e^\tau)}{U(e^\tau)},\qquad
z=cr^2U^{p-1},\qquad
d=-\frac{e^\tau c'(e^\tau)}{c(e^\tau)}.
\]
The radial equation is equivalent to the Emden--Fowler system
\begin{equation}\label{eq:efsystem}
\dot{y}=y^2-Ny+z,\qquad
\dot{z}=(2-d-(p-1)y)z.
\end{equation}
Indeed, $e^\tau U'(e^\tau)=-yU(e^\tau)$ gives
$-\Delta U=r^{-2}U(\dot{y}+y(N-y))$.
The second identity follows by logarithmic differentiation of $z$.

\begin{lemma}\label{lem:phaseconstruction}
Set
\[
h_-=1+\frac2N,\qquad h_*=\frac{p+1}{2},\qquad
h_+=p-\frac2N.
\]
Let $h\in C^\infty(\mathbb R)$ satisfy $\inf h>1$, with $h=h_-$ for all sufficiently negative arguments and $h=h_+$ for all sufficiently positive arguments. Define
\begin{equation}\label{eq:phaseparametrization}
\begin{aligned}
\tau(s)&=\int_0^s\frac{d\xi}{N(h(\xi)-1)},\\
y(\tau(s))&=\frac{Ne^s}{1+e^s},\\
z(\tau(s))&=h(s)y(\tau(s))(N-y(\tau(s))).
\end{aligned}
\end{equation}
Then
\begin{equation}\label{eq:inverseprofile}
U(r)=\exp\left(-\int_{-\infty}^{\log r}y(t)\,dt\right),
\qquad c(r)=r^{-2}U(r)^{1-p}z(\log r)
\end{equation}
defines a pair in $\mathcal A_{n,p}$. Both functions are smooth on $\mathbb R^n$, and
$\lim_{r\to\infty}r^NU(r)\in(0,\infty)$.
Its signed logarithmic defect is
\begin{equation}\label{eq:phasecontrol}
d(\tau(s))
=d_*+(h(s)-h_*)(2y-N)
-N\frac{h(s)-1}{h(s)}h'(s).
\end{equation}
\end{lemma}
\begin{proof}
The function $h$ is bounded above and bounded away from one. Thus $\tau(s)$ is a smooth increasing bijection of $\mathbb R$. Differentiating \eqref{eq:phaseparametrization} gives
\[
\frac{d}{ds}y(\tau(s))=\dot{y}(\tau(s))\tau'(s)=\frac{y(N-y)}N.
\]
Also,
\[
\frac{ds}{d\tau}=N(h-1),\qquad
\dot{y}=(h-1)y(N-y).
\]
Hence $\dot{y}+y(N-y)=hy(N-y)=z$. The first equation in \eqref{eq:efsystem} holds.

Near $s=-\infty$, $d\tau/ds=1/2$. Consequently there is $b_->0$ such that, for small $r$,
\[
y=\frac{Nb_-r^2}{1+b_-r^2},\qquad
U=(1+b_-r^2)^{-N/2}.
\]
This also proves convergence of the integral in \eqref{eq:inverseprofile}. Direct substitution yields
\[
c(r)=nNb_-(1+b_-r^2)^{-d_*}
\]
for small $r$. Thus $U$ and $c$ extend smoothly and positively across zero.

Near $s=+\infty$, $N(h_+-1)=\lambda$. There are constants $b_+,L_+>0$ such that, for large $r$,
\[
y=\frac{Nb_+r^\lambda}{1+b_+r^\lambda},\qquad
U=L_+(1+b_+r^\lambda)^{-N/\lambda}.
\]
Since $N(p-1)=\lambda+2$, formula \eqref{eq:inverseprofile} gives
\[
c(r)=N(N+\lambda)b_+L_+^{1-p}
r^{\lambda-2}(1+b_+r^\lambda)^{(2-\lambda)/\lambda}.
\]
The coefficient has a finite positive limit at infinity. The same formula gives $r^NU(r)\to L_+b_+^{-N/\lambda}\in(0,\infty)$. Together with positivity on the intermediate compact annulus, the two tail formulas give $0<c_-\le c\le c_+<\infty$.

The first identity in \eqref{eq:efsystem} and \eqref{eq:inverseprofile} imply
$-\Delta U=cU^p$ for $r>0$. Smoothness extends the equation to zero. To compute the logarithmic derivative, use
\[
\frac{\dot{z}}{z}
=N(h-1)\frac{h'}h+(h-1)(N-2y).
\]
Since $\log c(e^\tau)=-2\tau+(1-p)\log U(e^\tau)+\log z$,
\[
d=2-(p-1)y-(h-1)(N-2y)-N(h-1)\frac{h'}h.
\]
The identity $d_*=2-N(p-1)/2$ reduces this to \eqref{eq:phasecontrol}. Since $h'$ is bounded and compactly supported, \eqref{eq:phasecontrol} gives $\sup|d|<\infty$. Thus $(c,U)\in\mathcal A_{n,p}$.
\end{proof}

Along the curve of Lemma \ref{lem:phaseconstruction}, formula \eqref{eq:radialpohozaev} reads
\[
P(r)=\frac{r^NU(r)^2}{p+1}y(N-y)(h(s)-h_*),
\qquad r=e^{\tau(s)}.
\]
Thus $h=h_*$ is the zero level of the Pohozaev quantity.

\begin{proof}[Proof of Proposition~\ref{prop:scalarconstruction}]
Proposition~\ref{prop:radialendpoint} gives the strict lower bound.
We construct an approximating family with exponentially small defect excess.
Fix a smooth nondecreasing function $\vartheta:\mathbb R\to[0,1]$,
equal to zero on $(-\infty,0]$ and one on $[1,\infty)$.
Put $\kappa=n/4$ and $q_0=d_*/N$. For $T\ge2$, set
\[
q_T(s)=q_0\exp\left\{-\kappa\int_{-2T}^s
\vartheta(v+2T)\,dv\right\}
\bigl(1-\vartheta(s+T+1)\bigr).
\]
Then $q_T=q_0$ on $s\le-2T$ and $q_T=0$ on $s\ge-T$.
Differentiation gives
\begin{equation}\label{eq:exponentialtaper}
0\le-q_T'(s)\le\kappa q_T(s)+C e^{-\kappa T}.
\end{equation}
Indeed, the term containing $\vartheta'(s+T+1)$ is supported in
$[-T-1,-T]$. On that interval the integral in the exponent is at least
$T-2$. The constant in \eqref{eq:exponentialtaper} depends only on
$n,p$ and the fixed cutoff.

Define
\[
h_T(s)=
\begin{cases}
h_*+q_T(s),&s\le0,\\
h_*-q_T(-s),&s\ge0.
\end{cases}
\]
This function is smooth, nonincreasing, and equals $h_*$ on $[-T,T]$.
It equals $h_-$ on $s\le-2T$ and $h_+$ on $s\ge2T$.
Apply Lemma~\ref{lem:phaseconstruction} and denote the resulting pair by
$(c_T,U_T)$. Write
\[
b_T(s)=d_*+(h_T(s)-h_*)(2y-N).
\]
The signs of $h_T-h_*$ and $2y-N$ give
\[
\begin{cases}
2d_*y/N\le b_T\le d_*,&s\le0,\\
2d_*(N-y)/N\le b_T\le d_*,&s\ge0.
\end{cases}
\]
Since $h_T'\le0$, formula \eqref{eq:phasecontrol} gives $d_T\ge0$.
On each transition interval, $|s|\ge T\ge2$, whence
\[
|2y-N|\ge N/2,\qquad
N\frac{h_T-1}{h_T}\le\frac{2N}{n}.
\]
Writing $q=|h_T-h_*|$, we obtain from \eqref{eq:exponentialtaper}
\[
d_T-d_*
\le-\frac N2q+\frac{2N}{n}
\left(\frac n4q+Ce^{-\kappa T}\right)
\le Ce^{-\kappa T}.
\]
Outside the transitions $h_T'=0$ and $b_T\le d_*$. Therefore
\begin{equation}\label{eq:thresholdapproximation}
0\le-\frac{rc_T'(r)}{c_T(r)}
\le d_*+Ce^{-\kappa T},\qquad \kappa=\frac n4.
\end{equation}
The endpoint result gives $D(c_T)>d_*$.
Choosing $T\ge\max\{2,\kappa^{-1}\log(C/\varepsilon)\}$ proves
\eqref{eq:radialsharp} and the defect assertion.
Lemma~\ref{lem:phaseconstruction} gives smoothness and
the positive fast tail. Finally replace $c_T,U_T$ by
\[
c_T/c_T(0),\qquad c_T(0)^{1/(p-1)}U_T.
\]
This preserves the equation, defect and ratio, and gives $c_T(0)=1$.
\end{proof}

\begin{remark}\label{rem:contrast}
On $-T\le s\le T$, \eqref{eq:phasecontrol} gives $d=d_*$. The corresponding $\tau$ interval has length $2T/(2-d_*)$. Since $d\ge0$,
\[
\log\frac{c_T(0)}{c_T(\infty)}
=\int_{\mathbb R}d(\tau)\,d\tau
\ge\frac{2d_*T}{2-d_*}.
\]
After the normalization $c_T(0)=1$, this construction has $c_T(\infty)\to0$. Theorem \ref{thm:contrast} shows that unbounded coefficient ratios occur for every sequence whose defects approach $d_*$. Proposition \ref{prop:contrastconstruction} gives an upper bound for the ratio of this particular family.
\end{remark}

For example, if $n=3$ and $p=4$, Proposition \ref{prop:counterexample} gives a coefficient with defect one. Theorem \ref{thm:radialthreshold} gives
\[
D_{\mathrm{crit}}(3,4)=\frac12,
\]
and, for every $\varepsilon>0$, a smooth decreasing coefficient with a positive entire solution and defect in $(1/2,1/2+\varepsilon]$. The endpoint itself is excluded by \eqref{eq:radialpohozaevderivative}. This sharpness concerns the radial scalar subclass with equal coefficients. The unequal-power optimization is treated in Section~\ref{sec:systems}.

\begin{remark}\label{rem:radialcomparison}
For a prescribed coefficient, set
\[
G_c(r)=\frac{r^nc(r)}{p+1}
-\frac N2\int_0^r s^{n-1}c(s)\,ds.
\]
Then
\[
G_c'(r)=\frac{r^{n-1}c(r)}{p+1}
\left(d_*+\frac{rc'(r)}{c(r)}\right),\qquad
\lim_{r\downarrow0}\frac{G_c(r)}{r^n}
=\frac{d_*c(0)}{n(p+1)}>0.
\]
If $D(c)\le d_*$, then $G_c\ge0$ and $G_c\not\equiv0$.
This connects the defect condition with the sign functions used in
classical radial Pohozaev criteria; see~\cite{KNY,KYY}.
The nonexistence conclusion needed here follows directly from
Proposition~\ref{prop:radialendpoint}.
The construction above varies the coefficient while controlling its defect.
Its logarithmic derivative tends to zero at both ends and is negative
in between. Thus it does not satisfy the monotonicity condition on
$rc'/c$ in~\cite{Yanagida}.
No uniqueness for a prescribed coefficient is asserted.
\end{remark}

\subsection{Fixed coefficient ratios and limiting profiles}\label{sec:contrast}

The preceding family reaches the threshold only by allowing its
coefficient ratio to grow. We now show that this behavior is necessary
for every near-threshold sequence, and estimate the gap in terms of the
prescribed ratio. The construction supplies a complementary upper bound.

We retain $N,\mu,d_*,\lambda$ from the preceding subsections. All constants in this subsection may depend on $n,p$. Dependence on a prescribed ratio $M$ is indicated separately.

\begin{proposition}\label{prop:scalarquantlower}
There is $c=c(n,p)>0$ such that every $(c_0,U)\in\mathcal A_{n,p}$
with $\mathcal R(c_0)\le M$ satisfies
\begin{equation}\label{eq:scalarpolynomialgap}
D(c_0)-d_*\ge cM^{-\Theta_s},\qquad
\Theta_s=\frac n2+\frac n{d_*}-1.
\end{equation}
\end{proposition}
\begin{proof}
Put $a_0=\inf c_0$ and $\rho=(a_0U(0)^{p-1})^{-1/2}$.
Replacing $U(r)$ and $c_0(r)$ by
$U(\rho r)/U(0)$ and $c_0(\rho r)/a_0$ preserves the equation,
the defect and the coefficient ratio. Thus assume
$U(0)=1$ and $1\le c_0\le M$.
Write $D=D(c_0)>d_*$ and $F=r^D c_0$.
Then $dF\ge0$. The scalar Pohozaev identity gives
\[
(D-d_*)I=\int_0^\infty r^{n-D}U^{p+1}\,dF,\qquad
I=\int_0^\infty r^{n-1}c_0U^{p+1}\,dr.
\]
The flux formula and $U(\infty)=0$ give $-rU'(r)\le NU(r)$.
Integration by parts and \eqref{eq:radialdecay} yield
\[
\begin{aligned}
I&=\int_0^\infty r^{n-1}|U'|^2\,dr
\le N^2\int_0^\infty r^{n-3}U^2\,dr\\
&\le N^2\int_0^\infty r^{n-3}
\left(1+\frac{p-1}{2n}r^2\right)^{-2/(p-1)}\,dr
\le C(n,p).
\end{aligned}
\]
The last integral is finite because $2\mu>N$.
The upper coefficient bound gives
$U(r)\ge1-Mr^2/(2n)$.
Set $R=(n/M)^{1/2}$ and $r_0=R(2M)^{-1/D}$.
Then $U\ge1/2$ on $[0,R]$ and
$F(R)-F(r_0)\ge R^D/2$. Consequently
\[
(D-d_*)I\ge
cR^n(2M)^{-(n-D)_+/D}
\ge cM^{-n/2-(n/d_*-1)}.
\]
Division by $I\le C(n,p)$ proves \eqref{eq:scalarpolynomialgap}.
\end{proof}

\begin{proposition}\label{prop:contrastconstruction}
For the family constructed above there are $C_1,C_2>0$, depending only on
$n,p$, such that, for $T\ge2$,
\begin{equation}\label{eq:contrastfamily}
\begin{gathered}
d_*<D(c_T)\le d_*+C_1e^{-nT/4},\\
\left|\log\mathcal R(c_T)-\frac{4d_*}{a}T\right|\le C_2,
\qquad a=2-d_*.
\end{gathered}
\end{equation}
For each $0<\kappa<n/2$, a second family satisfies
\begin{equation}\label{eq:fixedcutcontrast}
0<D(c_{T,\kappa})-d_*\le C_\kappa e^{-\kappa T},\qquad
\log\mathcal R(c_{T,\kappa})=\frac{2d_*}{a}T+O_\kappa(1)
\end{equation}
for $T\ge T_0(n,p,\kappa)$.
Its central cutoff is fixed while the outer transitions recede.
In particular, for every $0<\rho<na/(4d_*)$, there are $C_\rho,M_\rho>0$
such that
\[
D_{\mathrm{crit}}(n,p;M)-d_*\le C_\rho M^{-\rho}
\qquad(M\ge M_\rho).
\]
\end{proposition}
\begin{proof}
For the original family the defect estimate is
\eqref{eq:thresholdapproximation}.
Since $c_T$ is nonincreasing,
\[
\log\mathcal R(c_T)=\int_{\mathbb R}d_T(\tau)\,d\tau.
\]
The exponential cutoff satisfies
\[
\int_{-2T}^{2T}|h_T-h_*|\,ds
=2\int_{-2T}^{-T}q_T(s)\,ds
\le2q_0(1+\kappa^{-1}),\qquad \kappa=n/4.
\]
Indeed, with $v=s+2T$, its exponential factor is at most
$e^{-\kappa(v-1)_+}$.
Write $b_T=d_*+(h_T-h_*)(2y-N)$.
The change of variable in \eqref{eq:phaseparametrization} gives
\[
d_T\frac{d\tau}{ds}
=\frac{b_T}{N(h_T-1)}-\frac{h_T'}{h_T}.
\]
The bounds $N(h_T-1)\ge\lambda>0$ and $|2y-N|\le N$ imply
\[
\left|\frac{b_T}{N(h_T-1)}-\frac{d_*}{a}\right|
\le C|h_T-h_*|.
\]
The derivative term has integral $\log(h_-/h_+)$.
The two constant-$h_T$ tails have total integral
\[
\left(d_*+\frac{2d_*}{\lambda}\right)\log(1+e^{-2T}).
\]
Integration on $[-2T,2T]$ therefore proves
\eqref{eq:contrastfamily}.

Fix $0<\kappa<n/2$. Choose $S\ge2$, depending only on $n,\kappa$, so that
$\tanh(S/2)>2\kappa/n$. For $T\ge S+2$, set
\[
q_{T,\kappa}(s)=q_0
\exp\left\{-\kappa\int_{-T}^s\vartheta(v+T)\,dv\right\}
\bigl(1-\vartheta(s+S+1)\bigr).
\]
It equals $q_0$ on $s\le-T$ and zero on $s\ge-S$.
Define $h_{T,\kappa}=h_*+q_{T,\kappa}(s)$ for $s\le0$, and
$h_{T,\kappa}=h_*-q_{T,\kappa}(-s)$ for $s\ge0$.
Then
\[
0\le-q_{T,\kappa}'\le\kappa q_{T,\kappa}+C_\kappa e^{-\kappa T},
\qquad
\int_{-T}^T|h_{T,\kappa}-h_*|\,ds\le C_\kappa.
\]
The cutoff derivative is supported in $[-S-1,-S]$, where the exponential
factor is at most $e^{-\kappa(T-S-2)}$.
On both transitions, $|s|\ge S$, so \eqref{eq:phasecontrol} gives
\[
d_{T,\kappa}-d_*
\le-N\tanh(S/2)|h_{T,\kappa}-h_*|
+\frac{2N}{n}\bigl(\kappa|h_{T,\kappa}-h_*|
+C_\kappa e^{-\kappa T}\bigr)
\le C_\kappa e^{-\kappa T}.
\]
The signed defect remains nonnegative, by the same sign argument as for
the base family. Integrating on $[-T,T]$, using the displayed $L^1$ bound,
gives $2d_*T/a+O_\kappa(1)$ for the logarithmic ratio.
The constant-phase tails contribute $O(e^{-T})$.
Proposition~\ref{prop:radialendpoint} gives strictness, proving
\eqref{eq:fixedcutcontrast}.

For $0<\rho<na/(4d_*)$, take $\kappa=2d_*\rho/a<n/2$ and
\[
T=\frac{a}{2d_*}(\log M-C_\kappa),
\]
with $C_\kappa$ enlarged to bound the ratio error.
For sufficiently large $M$, the resulting coefficient has ratio at most
$M$ and defect excess at most $C_\rho M^{-\rho}$.
\end{proof}

\begin{proof}[Completion of Theorem~\ref{thm:contrast}]
Proposition~\ref{prop:radialendpoint} and the construction in
Proposition~\ref{prop:scalarconstruction} prove the threshold and
nonattainment assertions. Propositions~\ref{prop:scalarquantlower}
and~\ref{prop:contrastconstruction} give the quantitative bounds: for
every $0<\rho<n(2-d_*)/(4d_*)$ there are $c_{n,p},C_\rho>0$ and
$M_\rho>1$ such that
\begin{equation}\label{eq:contrastupper}
c_{n,p}M^{-\Theta_s}\le D_{\mathrm{crit}}(n,p;M)-d_*
\le C_\rho M^{-\rho}\qquad(M\ge M_\rho).
\end{equation}
The lower bound \eqref{eq:scalarpolynomialgap} holds for every admissible pair
with ratio at most $M\ge1$. If $D(c_j)\to d_*$ and $\mathcal R(c_j)$ does not tend to infinity, a subsequence has $\mathcal R(c_j)\le M$ for some finite $M$. Formula \eqref{eq:scalarpolynomialgap} contradicts $D(c_j)\to d_*$. Therefore $\mathcal R(c_j)\to\infty$. Also, the classes in \eqref{eq:contrastinfimum} increase with $M$. Together with \eqref{eq:contrastupper}, this gives
\[
D_{\mathrm{crit}}(n,p;M)\downarrow d_*
\quad\hbox{as }M\to\infty.
\]
\end{proof}

For $n=3,p=4$, the bounds in \eqref{eq:contrastupper} read
\[
cM^{-13/2}\le D_{\mathrm{crit}}(3,4;M)-\tfrac12
\le C_\rho M^{-\rho},\qquad 0<\rho<\tfrac94.
\]
They do not determine the optimal exponent.
The fixed-cut family in \eqref{eq:fixedcutcontrast} improves the ratio bound.
The base family is retained below because it agrees exactly with the
critical profile on expanding annuli.

\begin{proposition}\label{prop:criticallimit}
For the base family in Proposition \ref{prop:contrastconstruction}, set
\[
a=2-d_*=\frac{N(p-1)}2,\qquad
V_T(r)=\frac{U_T(r)}{U_T(1)},\qquad
\widehat c_T(r)=c_T(r)U_T(1)^{p-1}.
\]
On the annulus $e^{-T/a}\le r\le e^{T/a}$, the identities
\begin{equation}\label{eq:criticalprofile}
V_T(r)=V_*(r):=\left(\frac2{1+r^a}\right)^{N/a},\qquad
\widehat c_T(r)=c_*(r):=\frac{N(N+a)}4r^{-d_*}
\end{equation}
hold exactly. In particular, $(V_T,\widehat c_T)\to(V_*,c_*)$ in $C^k$ on every compact annulus, for each integer $k\ge0$. The limit solves
\begin{equation}\label{eq:criticalweighted}
-\Delta V_*=\frac{N(N+a)}4|x|^{-d_*}V_*^p
\quad\hbox{on }\mathbb R^n\setminus\{0\},\qquad
p=\frac{n+2-2d_*}{n-2}.
\end{equation}
It is not a pair in $\mathcal A_{n,p}$.
\end{proposition}
\begin{proof}
The parametrization in \eqref{eq:phaseparametrization} fixes $\tau(0)=0$. On $-T\le s\le T$, one has $h_T=h_*$ and
\[
N(h_*-1)=a,\qquad \tau=s/a,\qquad
y=\frac{Nr^a}{1+r^a},\qquad
z=\frac{h_*N^2r^a}{(1+r^a)^2}.
\]
Integrating $\frac{d}{d\tau}\log V_T(e^\tau)=-y$ from $r=1$, with $V_T(1)=1$, gives the first identity in \eqref{eq:criticalprofile}. Since
\[
z=\widehat c_T r^2V_T^{p-1},\qquad
V_T^{p-1}=\frac4{(1+r^a)^2},\qquad
h_*N^2=N(N+a),
\]
the second identity follows. For any $0<r_0<r_1<\infty$, choose $T\ge a\max\{|\log r_0|,|\log r_1|\}$. Then \eqref{eq:criticalprofile} holds on $[r_0,r_1]$. The asserted $C^k$ convergence follows.

The equation for $V_T$ is $-\Delta V_T=\widehat c_TV_T^p$. Passing to the limit on compact annuli proves \eqref{eq:criticalweighted}; the exponent relation also follows from the definition of $d_*$. Since $0<d_*<1$, the coefficient $c_*$ is unbounded at zero and tends to zero at infinity. Moreover, $1<a<2$ and
\[
V_*''(r)=-2^{N/a}N(a-1)r^{a-2}+O(r^{2a-2})
\quad\hbox{as }r\downarrow0.
\]
Thus $V_*$ is not $C^2$ at zero. These facts exclude $(c_*,V_*)$ from $\mathcal A_{n,p}$.
\end{proof}

Proposition~\ref{prop:criticallimit} concerns the constructed family.
Corollary~\ref{cor:scalarsequencelimit} gives the same profile for arbitrary
near-threshold sequences with bounded normalized energy.

\section{Scalar stability, profiles and realization}\label{sec:scalarcompactness}

Throughout this section, \eqref{eq:radialrange} holds.
Section~\ref{sec:radial} determined the threshold and showed that
approaching it forces coefficient-ratio degeneration. We now study the
corresponding solutions. Local rigidity identifies the critical profile;
global stability singles out the least limiting energy. For a general
bounded-energy sequence we then extract separated profiles, identify
the diffuse energy and give criteria excluding it. The final two
subsections construct the possible amplitudes and remaining energies.

\subsection{Quantitative scalar convergence}

The normalization removes the free amplitude and radius before any
comparison of solutions. The decay estimate \eqref{eq:radialdecay}
and regularity at zero justify integration by parts, and imply that
$r^NU(r)^2$ vanishes at both endpoints. Thus, for
$(c,U)\in\mathcal A_{n,p}$, the following energy and attained maximum
are finite and positive:
\begin{equation}\label{eq:scalarpeakdata}
I=\int_0^\infty r^{n-1}cU^{p+1}\,dr
=\int_0^\infty r^{n-1}|U'|^2\,dr,\qquad
Z=\max_{r>0}r^NU(r)^2,\qquad H=I/Z.
\end{equation}
Choose a maximizing radius $r_0$ and define
\begin{equation}\label{eq:scalarpeaknormalization}
V(r)=\frac{U(r_0r)}{U(r_0)},\qquad
\widehat c(r)=r_0^2U(r_0)^{p-1}c(r_0r).
\end{equation}
Then $V(1)=1$, $r^NV(r)^2\le1$, and the normalized energy equals $H$.
Let $V_*,c_*$ be the profiles in \eqref{eq:criticalprofile}, with
$a=2-d_*$. Their critical energy is
\begin{equation}\label{eq:minimumscalarenergy}
\begin{aligned}
H_*&=\int_0^\infty r^{n-1}|V_*'|^2\,dr
=\int_0^\infty r^{n-1}c_*V_*^{p+1}\,dr\\
&=\frac{N(N+a)}a\,2^{2N/a}\int_0^1[t(1-t)]^{N/a}\,dt.
\end{aligned}
\end{equation}
The explicit evaluation is verified in the proof of
Theorem~\ref{thm:scalarminimum}.

To turn radial dilations into translations, set
\begin{equation}\label{eq:scalarlogvariables}
\begin{aligned}
\tau&=\log r,& W(\tau)&=e^{N\tau/2}V(e^\tau),&
K(\tau)&=e^{d_*\tau}\widehat c(e^\tau),\\
W_*(\tau)&=\cosh(a\tau/2)^{-N/a},&&
k_*&=\frac{N(N+a)}4.
\end{aligned}
\end{equation}
The factor defining $K$ removes the critical power from the coefficient,
so a critical homogeneous coefficient becomes constant. We use the
energy space $\mathcal H=H^1(\mathbb R)$ with
\begin{equation}\label{eq:scalarhilbertenergy}
\|W\|_{\mathcal H}^2=
\int_{\mathbb R}\left(|\dot W|^2+\frac{N^2}{4}|W|^2\right)\,d\tau.
\end{equation}
The change of variables and the vanishing endpoint terms give
$\|W\|_{\mathcal H}^2=H$ and $\|W_*\|_{\mathcal H}^2=H_*$.
For sequences, the subscript $j$ is applied to these definitions.
The first result controls the normalized pair on a fixed annulus.

\begin{proposition}\label{prop:scalarstability}
Assume \eqref{eq:radialrange} and let $(c,U)\in\mathcal A_{n,p}$.
Set $\varepsilon=D(c)-d_*$.
For every $L>0$, the normalization \eqref{eq:scalarpeaknormalization}
satisfies
\begin{equation}\label{eq:scalarprofilestability}
\|V-V_*\|_{C^2(I_L)}
+\|\widehat c-c_*\|_{C^0(I_L)}
\le C_{n,p,L}\varepsilon H,
\qquad I_L=[e^{-L},e^L].
\end{equation}
The norms are taken in the radial variable.
\end{proposition}

We first establish phase bounds which are uniform in the coefficient
ratio. They also control the energy on each fixed logarithmic interval.

\begin{lemma}\label{lem:scalarphasebounds}
Let $(c,U)\in\mathcal A_{n,p}$ with $D(c)\le D_0$.
Use the peak normalization above and set
\[
y=-\frac{e^\tau V'(e^\tau)}{V(e^\tau)},\qquad
z=e^{2\tau}\widehat c(e^\tau)V(e^\tau)^{p-1}
=KW^{p-1}.
\]
There is $C=C(n,p,D_0)$ such that
\[
0<y<N,\qquad 0<z\le C,\qquad
0<W\le1,\qquad |\dot{W}|\le\frac N2W,\qquad |\ddot{W}|\le CW.
\]
In particular, for each $L>0$,
\begin{equation}\label{eq:scalarlogenergybound}
\sup_{h\in\mathbb R}\frac1H
\int_{h-L}^{h+L}KW^{p+1}\,d\tau\le\frac{2CL}{H}.
\end{equation}
\end{lemma}
\begin{proof}
Set $F(r)=-r^{n-1}V'(r)$.
The flux formula gives $F'>0$, and \eqref{eq:radialdecay} gives
$V(r)=\int_r^\infty F(t)t^{-N-1}\,dt$.
Thus $V(r)>F(r)r^{-N}/N$, which proves $0<y<N$.
In logarithmic radius,
\[
\dot{y}=z-y(N-y),\qquad
\frac{\dot{z}}z=2-(p-1)y+\frac{e^\tau\widehat c'(e^\tau)}{\widehat c(e^\tau)}
\ge 2-(p-1)N-D_0.
\]
Put $C_0=((p-1)N+D_0-2)_+$.
For $0\le t\le1$, one has $z(\tau+t)\ge e^{-C_0}z(\tau)$, whereas
\[
\int_\tau^{\tau+1}z
=y(\tau+1)-y(\tau)+\int_\tau^{\tau+1}y(N-y)
\le N+N^2/4.
\]
This bounds $z$. The peak normalization and the identities
\[
\dot{W}=(N/2-y)W,\qquad \ddot{W}=(N^2/4-z)W
\]
give the other bounds. Since $KW^{p+1}=zW^2\le C$, integration proves
\eqref{eq:scalarlogenergybound}.
\end{proof}

\begin{proof}[Proof of Proposition~\ref{prop:scalarstability}]
Let $P$ be the scalar Pohozaev function \eqref{eq:radialpohozaev}
for $(\widehat c,V)$, and write $D=D(c)$.
The nonnegative measure
\[
\frac{r^{n-1}\widehat cV^{p+1}}{p+1}
\left(D+\frac{r\widehat c'}{\widehat c}\right)\,dr
\]
has total mass $\varepsilon H/(p+1)$ by
\eqref{eq:radialpohozaevderivative} and $P(0)=P(\infty)=0$.
Since $dP$ is this measure minus
$\varepsilon r^{n-1}\widehat cV^{p+1}\,dr/(p+1)$,
\[
|P(r)|\le\frac{\varepsilon H}{p+1}\qquad(r>0).
\]

Set $\tau=\log r$ and
\[
y=-\frac{e^\tau V'(e^\tau)}{V(e^\tau)},
\qquad z=\widehat c\,r^2V^{p-1},
\qquad Q=r^NV^2.
\]
The flux formula gives $0<y<N$.
The normalization gives $Q(1)=1$ and $y(0)=N/2$.
One has
\[
\frac P Q=\frac z{p+1}-\frac{y(N-y)}2,\qquad
\dot{y}=\frac{p-1}{2}y(N-y)+(p+1)\frac P Q.
\]
Also $|\frac{d}{d\tau}\log Q(e^\tau)|=|N-2y|\le N$. Thus
\[
Q(e^\tau)\ge e^{-N|\tau|},\qquad
\left|(p+1)\frac P Q\right|\le e^{NL}\varepsilon H
\quad(|\tau|\le L).
\]

The critical slope $y_*(\tau)=Ne^{a\tau}/(1+e^{a\tau})$,
where $a=2-d_*$, solves
\[
\dot{y}_*=\frac{p-1}{2}y_*(N-y_*),\qquad y_*(0)=N/2.
\]
The right side is Lipschitz on $[0,N]$.
Gronwall's inequality, applied on both sides of zero, gives
\[
\|y-y_*\|_{C^0([-L,L])}
+\|\dot{y}-\dot{y}_*\|_{C^0([-L,L])}
\le C_{n,p,L}\varepsilon H.
\]
Integrating $\frac{d}{d\tau}\log V(e^\tau)=-y$ with $V(1)=1$ bounds $V-V_*$.
The bounds $0<y<N$ give uniform upper and positive lower bounds for $V$
on $I_L$. The identities
\[
V'=-\frac{yV}{r},\qquad
V''=\frac V{r^2}(y^2+y-\dot{y}),\qquad
z=\frac{p+1}{2}y(N-y)+(p+1)\frac P Q
\]
therefore give the claimed $C^2$ estimate.
Finally $\widehat c=r^{-2}zV^{1-p}$ gives the coefficient estimate.
\end{proof}

\begin{corollary}\label{cor:scalarsequencelimit}
Let $(c_j,U_j)\in\mathcal A_{n,p}$, let
$\varepsilon_j=D(c_j)-d_*\to0$, and use the normalization
\eqref{eq:scalarpeakdata}--\eqref{eq:scalarpeaknormalization}.
If $\varepsilon_jH_j\to0$, then
\begin{equation}\label{eq:universalscalarprofile}
V_j\to V_*\quad\hbox{in }C^2_{\mathrm{loc}}((0,\infty)),\qquad
\widehat c_j\to c_*\quad\hbox{in }C^0_{\mathrm{loc}}((0,\infty)).
\end{equation}
In particular, these limits hold whenever $H_j$ stays bounded.
If $H_j\to\infty$, then for every $L>0$,
\[
\sup_{h\in\mathbb R}\frac1{I_j}
\int_{e^{h-L}}^{e^{h+L}}r^{n-1}c_jU_j^{p+1}\,dr
\longrightarrow0.
\]
The last assertion concerns the energy fraction; it does not imply
local convergence of $V_j$ to zero.
\end{corollary}
\begin{proof}
The local limits follow from \eqref{eq:scalarprofilestability}.
The final assertion follows from \eqref{eq:scalarlogenergybound}.
Spatial dilation translates the logarithmic intervals and leaves their
energy fractions unchanged.
\end{proof}

\subsection{Global scalar stability and minimum energy}

Local convergence does not control energy outside a fixed annulus.
The next estimate compares the full Dirichlet distance with the excess
over the critical energy. It identifies the least limiting energy and
the condition under which local convergence becomes global.

\begin{theorem}\label{thm:scalarminimum}
Assume \eqref{eq:radialrange} and let $(c,U)\in\mathcal A_{n,p}$.
Use the preceding peak normalization and set
\[
\varepsilon=D(c)-d_*,\qquad \eta=\varepsilon H.
\]
There is $C>0$, depending only on $n,p$, such that
\begin{equation}\label{eq:globalscalarstability}
\left|\int_0^\infty r^{n-1}|V'-V_*'|^2\,dr-(H-H_*)\right|
\le C\min\{\eta,1\}.
\end{equation}
In particular,
\[
H\ge H_*-C\eta,\qquad
\int_0^\infty r^{n-1}|V'-V_*'|^2\,dr
\le(H-H_*)_++C\eta.
\]
Every near-threshold sequence satisfies
\begin{equation}\label{eq:minimumenergyliminf}
\liminf_{j\to\infty}H_j\ge H_*.
\end{equation}
The base family in Proposition~\ref{prop:contrastconstruction} satisfies
$H_T\to H_*$.
If
\[
\Phi(\varepsilon)=
\inf_{\substack{(c,U)\in\mathcal A_{n,p}\\
d_*<D(c)\le d_*+\varepsilon}}H,
\]
then there are $c,C,\varepsilon_0>0$, depending only on $n,p$, such that
\begin{equation}\label{eq:minimumenergyvariational}
c\varepsilon\le H_*-\Phi(\varepsilon)\le C\varepsilon
\qquad(0<\varepsilon<\varepsilon_0).
\end{equation}
The linear error in \eqref{eq:globalscalarstability} is optimal in order.
It cannot be replaced by $o(\varepsilon H)$, even for nonincreasing smooth
coefficients and sequences with $H\to H_*$.
Proposition~\ref{prop:linearsharpness} gives such a sequence.
If $D(c_j)\to d_*$ and $H_j\to H_*$, the peak-normalized pairs satisfy
\begin{equation}\label{eq:strongscalarcompactness}
\int_0^\infty r^{n-1}|V_j'-V_*'|^2\,dr\longrightarrow0
\end{equation}
and
\begin{equation}\label{eq:nonlinearenergyconvergence}
\int_0^\infty
\left|r^{n-1}\widehat c_jV_j^{p+1}-r^{n-1}c_*V_*^{p+1}\right|\,dr
\longrightarrow0.
\end{equation}
\end{theorem}

\begin{proof}[Proof of Theorem~\ref{thm:scalarminimum}]
All constants in this proof depend only on $n,p$.
The explicit profile solves \eqref{eq:criticalweighted}.
Its endpoint terms in integration by parts vanish. Hence
\[
H_*=
\int_0^\infty r^{n-1}|V_*'|^2\,dr
=\int_0^\infty r^{n-1}c_*V_*^{p+1}\,dr.
\]
The substitution $t=r^a/(1+r^a)$ gives the expression in
\eqref{eq:minimumscalarenergy}.

\smallskip\noindent\emph{The global estimate.}
Set
\[
W(\tau)=e^{N\tau/2}V(e^\tau),\qquad
W_*(\tau)=\cosh(a\tau/2)^{-N/a},\qquad
k_*=\frac{N(N+a)}4.
\]
The peak normalization and the flux bound give
\[
0<W\le1,\qquad |\dot{W}|\le\frac N2W,\qquad
-\ddot{W}_*+\frac{N^2}{4}W_*=k_*W_*^p.
\]
The endpoint terms vanish, since $W,\dot{W}$ are bounded and $W_*,\dot{W}_*$
decay exponentially. Integration by parts therefore yields
\[
\begin{aligned}
\int_0^\infty r^{n-1}V'V_*'\,dr-H_*
&=\int_{\mathbb R}
\left(\dot{W}-\frac N2W\right)
\left(\dot{W}_*-\frac N2W_*\right)\,d\tau-H_*\\
&=k_*\int_{\mathbb R}(W-W_*)W_*^p\,d\tau.
\end{aligned}
\]
Since $0<W\le1$ and $W_*^p$ is integrable, the absolute value of this
expression is bounded by a constant depending only on $n,p$.
It is therefore enough to prove the bound $C\eta$ for
$0<\eta\le\eta_0$, where $\eta_0>0$ will be chosen below.

Use the variables $y,Q$ in the proof of
Proposition~\ref{prop:scalarstability}, and put $\sigma=\max\{N,a\}$.
Since $a\in(1,2)$ and $N$ is an integer, $N\ne a$.
The Lipschitz constant of $\frac{p-1}{2}y(N-y)$ on $[0,N]$ is $a$.
The bounds $|P|\le\eta/(p+1)$ and $Q(e^\tau)\ge e^{-N|\tau|}$ give
\[
|y(\tau)-y_*(\tau)|
\le\eta e^{a|\tau|}\int_0^{|\tau|}e^{(N-a)v}\,dv
\le C\eta e^{\sigma|\tau|}.
\]
Integration of $\frac{d}{d\tau}\log Q(e^\tau)=N-2y$ gives
\[
|\log Q(e^\tau)-\log Q_*(e^\tau)|
\le C\eta e^{\sigma|\tau|}.
\]
Choose a fixed small $c_0>0$ and set
$L=\sigma^{-1}\log(c_0/\eta)$.
For $\eta\le\eta_0(n,p)$, one has $L\ge1$ and
\[
|W-W_*|\le C\eta e^{\sigma|\tau|}W_*
\quad(|\tau|\le L),\qquad W(\pm L)\le Ce^{-NL/2}.
\]
Since $N(p+1)/2=N+a>\sigma$,
\[
\int_{-L}^L|W-W_*|W_*^p\,d\tau
\le C\eta\int_{\mathbb R}e^{\sigma|\tau|}W_*^{p+1}\,d\tau
\le C\eta.
\]
The differential inequality $|\dot{W}|\le NW/2$ also gives
\[
W(\tau)\le Ce^{-NL}e^{N|\tau|/2}\qquad(|\tau|\ge L).
\]
Using $Np/2=N/2+a$, we obtain
\[
\begin{aligned}
\int_{|\tau|>L}|W-W_*|W_*^p\,d\tau
&\le C e^{-NL}\int_L^\infty e^{-a t}\,dt
+C\int_L^\infty e^{-(N+a)t}\,dt\\
&\le C e^{-(N+a)L}\le C\eta.
\end{aligned}
\]
Consequently
\[
\left|\int_0^\infty r^{n-1}V'V_*'\,dr-H_*\right|\le C\eta.
\]
The identity
\[
\int_0^\infty r^{n-1}|V'-V_*'|^2\,dr
=H+H_*-2\int_0^\infty r^{n-1}V'V_*'\,dr
\]
proves \eqref{eq:globalscalarstability} for small $\eta$.
The preceding uniform bound gives the same estimate for larger $\eta$,
after increasing $C$. Its two consequences follow by nonnegativity of
the squared distance.

\smallskip\noindent\emph{The limiting minimum.}
If $\liminf_jH_j=\infty$, \eqref{eq:minimumenergyliminf} is immediate.
Otherwise choose a subsequence attaining a finite lower limit.
Then $\varepsilon_jH_j\to0$, and \eqref{eq:globalscalarstability} gives
$\liminf_jH_j\ge H_*$.

For the base family in Proposition~\ref{prop:contrastconstruction}, the slope
$y(s)=Ne^s/(1+e^s)$ crosses $N/2$ exactly at $s=0$.
Thus the maximizing radius is one. With $V_T=U_T/U_T(1)$, put
$Q_T(s)=e^{N\tau_T(s)}V_T(e^{\tau_T(s)})^2$.
Formula \eqref{eq:phaseparametrization} gives
\[
Q_T(s)=\exp\left\{
\int_0^s\frac{1-e^v}{(1+e^v)(h_T(v)-1)}\,dv\right\}.
\]
Since $\lambda\le N(h_T-1)\le2$,
\[
0<Q_T(s)\le\cosh(s/2)^{-N}\le2^Ne^{-N|s|/2},
\qquad 0<\tau_T'(s)\le1/\lambda.
\]
For each fixed $s$, one has $h_T(s)\to h_*$ and $\tau_T(s)\to s/a$.
The representation
\[
H_T=\int_{\mathbb R}Q_T(s)y(s)^2\tau_T'(s)\,ds
\]
and dominated convergence give $H_T\to H_*$.
This family proves $\Phi(\varepsilon)\le H_*$ for every $\varepsilon>0$.
For the reverse estimate in \eqref{eq:minimumenergyvariational},
consider a pair with $d_*<D(c)\le d_*+\varepsilon$.
If $H>H_*+1$, the estimate is immediate.
Otherwise $\eta\le\varepsilon(H_*+1)$, so
\eqref{eq:globalscalarstability} gives
$H\ge H_*-C\varepsilon$ for sufficiently small $\varepsilon$.
Taking the infimum proves the lower bound for $\Phi$ in
\eqref{eq:minimumenergyvariational}. Its strict upper bound and the
optimality of the linear stability error are proved in
Proposition~\ref{prop:linearsharpness} below.

\smallskip\noindent\emph{Strong convergence.}
If $H_j\to H_*$, then $\varepsilon_jH_j\to0$ and
\eqref{eq:globalscalarstability} proves \eqref{eq:strongscalarcompactness}.
The nonlinear densities converge locally uniformly by
\eqref{eq:universalscalarprofile}.
They are nonnegative, and their integrals converge to $H_*$, the integral
of the limiting density. Scheff\'e's lemma proves
\eqref{eq:nonlinearenergyconvergence}.
\end{proof}

\begin{proposition}\label{prop:linearsharpness}
For fixed $n,p$ in \eqref{eq:radialrange}, there are pairs
$(c_t,U_t)\in\mathcal A_{n,p}$, $t\downarrow0$, with nonincreasing smooth
coefficients and positive fast tails, such that their peak normalizations
satisfy
\begin{equation}\label{eq:linearsharpnessfamily}
\begin{aligned}
\varepsilon_t:=D(c_t)-d_*&=m_\chi t+O(t^2),\\
H_t&=H_*-\Gamma_\chi t+O(t^2),\\
\|W_t-W_*\|_{\mathcal H}^2&=t^2\|\psi\|_{\mathcal H}^2+O(t^3),
\end{aligned}
\end{equation}
where $m_\chi,\Gamma_\chi>0$ and $\psi\ne0$ are specified below.
Consequently
\begin{equation}\label{eq:linearerroroptimality}
\lim_{t\downarrow0}
\frac{\|W_t-W_*\|_{\mathcal H}^2-(H_t-H_*)}
{\varepsilon_tH_t}
=\frac{\Gamma_\chi}{m_\chi H_*}>0.
\end{equation}
Both bounds in \eqref{eq:minimumenergyvariational} hold.
\end{proposition}
\begin{proof}
Fix an even function $\chi\in C_c^\infty((-1,1))$ with
$\chi(0)>0$, $\chi\ge0$, and $\chi'\le0$ on $(0,\infty)$.
Write $a=N(p-1)/2$, $h_*=(p+1)/2$, and
$y(s)=Ne^s/(1+e^s)$.

\smallskip\noindent\emph{A compact phase perturbation.}
First use the auxiliary phase $h^t(s)=h_*+t\chi(s)$ on the whole line.
Define
\[
\begin{gathered}
\tau_t(s)=\int_0^s\frac{dv}{a+Nt\chi(v)},\qquad
Y_t(\tau)=y(\tau_t^{-1}(\tau)),\\
w_t(\tau)=\exp\left(-\int_0^\tau(Y_t(v)-N/2)\,dv\right).
\end{gathered}
\]
Then $w_0=W_*$, $w_t(0)=1$, and its unique maximum is at zero.
Outside a fixed compact interval in $\tau$, $h^t=h_*$.
On each of the two tails one therefore has
\[
w_t(\tau)=C_\pm(t)W_*(\tau-b_\pm(t)),
\]
with positive $C_\pm$ and smooth $C_\pm,b_\pm$ for small $t$.
On the remaining compact interval, smooth dependence follows from the
displayed integral formulas. Differentiation up to order two in $t$ gives
an integrable exponential bound on both tails. Hence
\[
\|w_t-W_*-t\psi\|_{\mathcal H}\le Ct^2,
\qquad \psi=\left.\partial_t w_t\right|_{t=0}.
\]
Implicit differentiation of $\tau_t(\tau_t^{-1}(\tau))=\tau$ yields
\[
v(\tau):=\left.\partial_tY_t(\tau)\right|_{t=0}
=\frac Na y'(a\tau)\int_0^{a\tau}\chi(s)\,ds,
\qquad
\psi(\tau)=-W_*(\tau)\int_0^\tau v(z)\,dz.
\]
Here $y'$ denotes differentiation in $s$.
The function $v$ is odd and strictly positive for $\tau>0$.
Thus $\psi<0$ for $\tau\ne0$ and $\psi\in\mathcal H$.
The critical equation gives
\[
\Gamma_\chi:=-2\langle W_*,\psi\rangle_{\mathcal H}
=-2k_*\int_{\mathbb R}\psi W_*^p\,d\tau>0.
\]
It follows that
\[
\|w_t\|_{\mathcal H}^2=H_*-\Gamma_\chi t+O(t^2),
\qquad
\|w_t-W_*\|_{\mathcal H}^2=t^2\|\psi\|_{\mathcal H}^2+O(t^3).
\]
The auxiliary phase need not produce an admissible coefficient at the
endpoints. We next restore the bounded-coefficient tails.

\smallskip\noindent\emph{Admissible tails and defect control.}
Use the base phase $h_T$ of Proposition~\ref{prop:contrastconstruction} and set
\[
h_T^t=h_T+t\chi,\qquad T\ge2.
\]
For small fixed $t>0$, one has $h_+\le h_T^t\le h_-$.
Lemma~\ref{lem:phaseconstruction} gives smooth admissible pairs with
positive fast tails. Their peak-normalized logarithmic profiles are
denoted by $W_{T,t}$.
Outside $[-1,1]$ the signed defect is the base defect.
Inside that interval, \eqref{eq:phasecontrol} gives, uniformly in $s,T$,
\[
d_{T,t}-d_*
=tF_\chi(s)+O(t^2),\qquad
F_\chi(s)=(2y(s)-N)\chi(s)-\frac a{h_*}\chi'(s).
\]
The function $F_\chi$ is odd and positive wherever $s>0$ and $\chi(s)>0$.
Set
\[
m_\chi=\max_{s\in\mathbb R}F_\chi(s)>0.
\]
The core defect stays positive for small $t$; the base defect is
nonnegative elsewhere. Thus the coefficients remain nonincreasing, and
\[
\left|D(c_{T,t})-d_*-m_\chi t\right|
\le C(t^2+e^{-nT/4}).
\]

For each fixed $t$, the phases agree with $h^t$ on $[-T,T]$.
Consequently $W_{T,t}\to w_t$ in local $C^1$.
In the $s$ variable, write $Q_{T,t}=W_{T,t}(\tau_{T,t}(s))^2$.
The bounds on $h_T^t$ give
\[
Q_{T,t}(s)\le\cosh(s/2)^{-N},\qquad
0<\tau_{T,t}'(s)\le\lambda^{-1},\qquad
\lambda=N(p-1)-2.
\]
Hence the energy integrands
$Q_{T,t}y^2\tau_{T,t}'$ have a common integrable bound.
Dominated convergence gives
$\|W_{T,t}\|_{\mathcal H}^2\to\|w_t\|_{\mathcal H}^2$.
Local convergence identifies every weak $\mathcal H$ limit as $w_t$.
Convergence of the norms therefore gives strong $\mathcal H$ convergence.

For each small $t$, choose $T=T(t)$ so large that
\[
\|W_{T(t),t}-w_t\|_{\mathcal H}\le t^2,
\qquad e^{-nT(t)/4}\le t^2.
\]
Set $W_t=W_{T(t),t}$ and take its admissible radial pair.
Only the auxiliary curve $w_t$ was differentiated in $t$.
The choice $T(t)$ need not be continuous. The two error bounds above
suffice for the expansions of the admissible family.
The preceding estimates prove \eqref{eq:linearsharpnessfamily}, and then
\eqref{eq:linearerroroptimality}.
The $\mathcal H$ distance equals the radial Dirichlet distance under the
Emden--Fowler transformation, since the endpoint terms vanish.

Finally, for a prescribed sufficiently small $\varepsilon>0$, take
$t=\varepsilon/(2m_\chi)$. Formula \eqref{eq:linearsharpnessfamily} gives
\[
0<D(c_t)-d_*\le\varepsilon,
\qquad H_t\le H_*-\frac{\Gamma_\chi}{4m_\chi}\varepsilon.
\]
This proves $\Phi(\varepsilon)\le H_*-c\varepsilon$.
The opposite bound was proved from \eqref{eq:globalscalarstability}.
The cutoff $\chi$ is fixed once and for all, so the constants depend only
on $n,p$.
\end{proof}

\subsection{Scalar profiles and diffuse energy}

When $H_j$ has a finite limit larger than $H_*$, one peak need not carry
all of the energy. In logarithmic radius this leads to a sequence of
translation extractions. Local rigidity identifies each extracted
profile, but a uniformly small remainder can still retain positive
energy over expanding intervals. The theorem also tracks the nonlinear
energy through the values of $K_j$; this distinguishes the energy of
the profiles from the diffuse remainder.

\begin{theorem}\label{thm:scalarprofiles}
Assume \eqref{eq:radialrange}, $D(c_j)\to d_*$ and $H_j\to h<\infty$.
Use any peak normalization. After extraction, there are
$m\in\mathbb N\cup\{\infty\}$, amplitudes $0<A_\ell\le1$ and centers
$\tau_{j,\ell}$ such that $A_1=1$, $\tau_{j,1}=0$, and
\[
|\tau_{j,\ell}-\tau_{j,k}|\longrightarrow\infty\qquad(\ell\ne k).
\]
For each $\ell\le m$,
\begin{equation}\label{eq:translatedcriticalprofiles}
\begin{aligned}
W_j(\tau_{j,\ell}+\cdot)&\longrightarrow A_\ell W_*
&&\text{in }C^2_{\mathrm{loc}}(\mathbb R),\\
K_j(\tau_{j,\ell}+\cdot)&\longrightarrow k_*A_\ell^{1-p}
&&\text{in }C^0_{\mathrm{loc}}(\mathbb R).
\end{aligned}
\end{equation}
Set $R_j^L=W_j-\sum_{\ell=1}^L A_\ell W_*(\cdot-\tau_{j,\ell})$.
For every $L\le m$,
\begin{equation}\label{eq:profileenergysplitting}
H_j=H_*\sum_{\ell=1}^L A_\ell^2+\|R_j^L\|_{\mathcal H}^2+o_j(1).
\end{equation}
Moreover, for every $r>2$,
\begin{equation}\label{eq:profileremaindervanishing}
\lim_{L\uparrow m}\limsup_{j\to\infty}
\left(\|R_j^L\|_{L^\infty(\mathbb R)}+\|R_j^L\|_{L^r(\mathbb R)}\right)=0.
\end{equation}
If $m<\infty$, the limit $L\uparrow m$ means $L=m$.
The remaining energy is
\begin{equation}\label{eq:diffuseenergydefinition}
h_{\mathrm{diff}}=h-H_*\sum_{\ell=1}^m A_\ell^2
=\lim_{L\uparrow m}\lim_{j\to\infty}\|R_j^L\|_{\mathcal H}^2\ge0.
\end{equation}
After the profile neighborhoods are removed, both the Dirichlet and nonlinear
energies vanish uniformly on fixed logarithmic intervals, in the precise
sense of \eqref{eq:diffuselocalvanishing} below.

The nonlinear energy also identifies the coefficient values at its
concentration scales. Let $[0,\infty]$ denote the compactification obtained
by adding the endpoint $\infty$. For every $f\in C([0,\infty])$,
\begin{equation}\label{eq:coefficientenergydistribution}
\lim_{j\to\infty}\int_{\mathbb R}f(K_j)K_jW_j^{p+1}\,d\tau
=H_*\sum_{\ell=1}^m A_\ell^2 f(k_*A_\ell^{1-p})
+h_{\mathrm{diff}}f(\infty).
\end{equation}
In particular,
\begin{equation}\label{eq:coefficientescape}
h_{\mathrm{diff}}
=\lim_{M\to\infty}\limsup_{j\to\infty}
\int_{\{K_j>M\}}K_jW_j^{p+1}\,d\tau.
\end{equation}
Thus the remainder vanishes in $\mathcal H$ after all profiles are removed
if and only if the right side of \eqref{eq:coefficientescape} is zero.
\end{theorem}

\begin{proof}[Proof of Theorem~\ref{thm:scalarprofiles}]
Write $\varepsilon_j=D(c_j)-d_*$ and $\eta_j=\varepsilon_jH_j\to0$.
In logarithmic radius the equation and the Pohozaev identity give
\[
\begin{gathered}
\ddot{W}_j-\frac{N^2}{4}W_j+K_jW_j^p=0,\qquad
P_j=\frac{|\dot{W}_j|^2}{2}-\frac{N^2}{8}W_j^2
+\frac{K_j}{p+1}W_j^{p+1},\\
|P_j|\le\frac{\eta_j}{p+1},\qquad
\dot{K}_j+\varepsilon_jK_j\ge0,\qquad
\int_{\mathbb R}(\dot{K}_j+\varepsilon_jK_j)W_j^{p+1}\,d\tau=\eta_j.
\end{gathered}
\]
Lemma~\ref{lem:scalarphasebounds} gives
\[
0<W_j\le1,\qquad |\dot{W}_j|\le\frac N2W_j,\qquad
0<K_jW_j^{p-1}\le C,\qquad |\ddot{W}_j|\le CW_j.
\]
On $[-R,R]$, integration of $\frac{d}{d\tau}(W_j^2)$ gives
\[
\int_{-R}^R\left(|\dot{W}_j|^2+\frac{N^2}{4}W_j^2\right)
=\int_{-R}^R\left(\dot{W}_j-\frac N2W_j\right)^2
+\frac N2[W_j^2]_{-R}^R.
\]
Since $W_j(\pm\infty)=0$ and the radial Dirichlet energy is finite, letting
$R\to\infty$ proves $W_j\in\mathcal H$ and
$\|W_j\|_{\mathcal H}^2=H_j$.

\smallskip\noindent\emph{Identification of translated limits.}
Suppose $W_j(t_j)\ge b>0$. On $[-R,R]$,
\[
W_j(t_j+\tau)\ge be^{-NR/2}.
\]
The preceding bounds give a subsequence converging in local $C^1$ to a
positive function $W$, weakly in $\mathcal H$.
Set $\widetilde K_j(\tau)=K_j(t_j+\tau)$.
These coefficients are locally uniformly bounded. Moreover,
\[
\begin{aligned}
0&\le e^{\varepsilon_jR}\widetilde K_j(R)
-e^{-\varepsilon_jR}\widetilde K_j(-R)\\
&\le e^{\varepsilon_jR}b^{-p-1}e^{N(p+1)R/2}\eta_j.
\end{aligned}
\]
Thus $\widetilde K_j\to k$ locally uniformly, for some $k\ge0$.
The equations give local $C^2$ convergence and
\[
\ddot{W}-\frac{N^2}{4}W+kW^p=0,\qquad
\frac{|\dot{W}|^2}{2}-\frac{N^2}{8}W^2+\frac{k}{p+1}W^{p+1}=0.
\]
If $k=0$, the only $H^1(\mathbb R)$ solution of the linear equation is
zero, a contradiction. Hence $k>0$.
An $H^1(\mathbb R)$ function with bounded derivative tends to zero at
both ends. The positive function $W$ therefore attains a maximum $A$.
At a maximizing point the energy identity gives
$k=k_*A^{1-p}$. Uniqueness for the autonomous initial-value problem gives
\[
W(\tau)=A W_*(\tau-t_0),\qquad 0<A\le1.
\]
Replacing $t_j$ by $t_j+t_0$ centers this profile at zero.

\smallskip\noindent\emph{Extraction.}
The peak normalization and Corollary~\ref{cor:scalarsequencelimit} give
the first profile $W_*$ at $\tau_{j,1}=0$.
Suppose $L$ profiles have been extracted with pairwise divergent centers.
Translations of a fixed $H^1$ function tend weakly to zero when their
centers diverge. Consequently
\[
\left\langle W_j,A_\ell W_*(\cdot-\tau_{j,\ell})\right\rangle_{\mathcal H}
\longrightarrow A_\ell^2H_*,
\]
and the pairwise inner products of distinct profiles tend to zero.
Expansion of the squared norm proves \eqref{eq:profileenergysplitting}.
In particular, $H_*\sum_{\ell=1}^L A_\ell^2\le h$.

Pass to a further subsequence on which
$b_L=\lim_j\|R_j^L\|_\infty$ exists.
The negative part of $R_j^L$ tends uniformly to zero.
Indeed, on fixed neighborhoods of the $L$ centers this follows from local
profile convergence and separation. Outside those neighborhoods it follows
from $W_j\ge0$ and the exponential tails of the subtracted profiles.
If $b_L>0$, choose $t_j$ with $R_j^L(t_j)\ge b_L/2$.
Local convergence at the old centers forces
$|t_j-\tau_{j,\ell}|\to\infty$ for every $\ell\le L$.
Hence $W_j(t_j)\ge b_L/2+o(1)$, and the preceding argument extracts a
new critical profile of amplitude $A_{L+1}\ge b_L/2$.

If this process stops, its final remainder has vanishing uniform norm.
Otherwise perform nested extraction. At step $L$, retain all convergence
statements for the first $L$ profiles and the limit defining $b_L$.
Choose the diagonal index so that the first $L$ local errors on $[-L,L]$,
the first $L$ norm-splitting errors, and the errors in the first $L$
limits $b_k$ are at most $L^{-1}$. Require all first $L$ center distances
to exceed $L$. Each requirement holds on the corresponding nested
subsequence. Thus every fixed finite collection of limits is preserved,
and the diagonal subsequence gives
\[
\sum_{\ell=1}^\infty A_\ell^2\le h/H_*,\qquad
0\le b_L\le2A_{L+1}\longrightarrow0.
\]
This proves uniform vanishing in \eqref{eq:profileremaindervanishing}.
The norm splitting bounds the remainder in $\mathcal H$, uniformly in
$L$ after taking $\limsup_j$.
For $r>2$, use
\[
\|R_j^L\|_r^r\le\|R_j^L\|_\infty^{r-2}\|R_j^L\|_2^2.
\]
This proves the remaining assertion in
\eqref{eq:profileremaindervanishing} and the energy formula
\eqref{eq:diffuseenergydefinition}.

\smallskip\noindent\emph{The energy outside the profiles.}
Define
\[
\begin{aligned}
e_j^{D}&=\left(\dot{W}_j-\frac N2W_j\right)^2,&
e_j^{P}&=K_jW_j^{p+1},\\
E_{j,L,R}&=\mathbb R\setminus
\bigcup_{\ell=1}^L[\tau_{j,\ell}-R,\tau_{j,\ell}+R].
\end{aligned}
\]
Both energy densities have total integral $H_j$ and satisfy
$0\le e_j^D+e_j^P\le CW_j^2$.
For fixed $L$, local convergence in
\eqref{eq:translatedcriticalprofiles} and separation give, for $X=D,P$,
\[
\lim_{R\to\infty}\lim_{j\to\infty}
\int_{E_{j,L,R}}e_j^X\,d\tau
=h-H_*\sum_{\ell=1}^L A_\ell^2.
\]
On $E_{j,L,R}$,
\[
W_j\le |R_j^L|+C\left(\sum_{\ell=1}^L A_\ell\right)e^{-NR/2}.
\]
Therefore, for every fixed $S>0$,
\begin{equation}\label{eq:diffuselocalvanishing}
\lim_{L\uparrow m}\lim_{R\to\infty}\limsup_{j\to\infty}
\sup_{t\in\mathbb R}
\int_{[t-S,t+S]\cap E_{j,L,R}}(e_j^D+e_j^P)\,d\tau=0.
\end{equation}
The total remaining energy tends to $h_{\mathrm{diff}}$ for either density.

\smallskip\noindent\emph{The coefficient distribution of nonlinear energy.}
Set $k_\ell=k_*A_\ell^{1-p}$ and keep $E_{j,L,R}$ from the preceding part.
First take $f\in C_c([0,\infty))$, with support in $[0,M]$.
Since $\int W_j^2\le4H_j/N^2$, one has
\[
\left|\int_{E_{j,L,R}}f(K_j)K_jW_j^{p+1}\,d\tau\right|
\le \frac{4M H_j}{N^2}\|f\|_\infty
\left(\sup_{E_{j,L,R}}W_j\right)^{p-1}.
\]
The remainder estimate and the exponential profile tails imply
\[
\lim_{L\uparrow m}\lim_{R\to\infty}\limsup_{j\to\infty}
\sup_{E_{j,L,R}}W_j=0.
\]
On each of the first $L$ profile intervals, local convergence gives
\[
\lim_{R\to\infty}\lim_{j\to\infty}
\int_{\tau_{j,\ell}-R}^{\tau_{j,\ell}+R}
f(K_j)K_jW_j^{p+1}\,d\tau
=f(k_\ell)k_\ell A_\ell^{p+1}\int_{\mathbb R}W_*^{p+1}
=H_*A_\ell^2f(k_\ell).
\]
The intervals are disjoint for large $j$ at fixed $L,R$.
First let $j\to\infty$, then $R\to\infty$, and finally $L\uparrow m$.
No bound uniform in $L$ is required for the finite sum of profile tails.
These iterated limits prove
\[
\lim_{j\to\infty}\int_{\mathbb R}f(K_j)K_jW_j^{p+1}\,d\tau
=H_*\sum_{\ell=1}^m A_\ell^2f(k_\ell)
\quad(f\in C_c([0,\infty))).
\]
Both sides are bounded by a constant times $\|f\|_\infty$.
Uniform approximation extends this identity to every continuous $f$
vanishing at infinity. For $f\in C([0,\infty])$, apply it to
$f-f(\infty)$ and use $\int K_jW_j^{p+1}=H_j\to h$.
This proves \eqref{eq:coefficientenergydistribution}.

Equivalently, the finite measures
$(K_j)_\#(K_jW_j^{p+1}\,d\tau)$ converge weakly on $[0,\infty]$ to
\[
H_*\sum_{\ell=1}^m A_\ell^2\delta_{k_\ell}
+h_{\mathrm{diff}}\delta_\infty.
\]
For $M\notin\{k_\ell:\ell\le m\}$, the boundary of $(M,\infty]$ has
zero limiting mass. Therefore
\[
\lim_{j\to\infty}\int_{\{K_j>M\}}K_jW_j^{p+1}\,d\tau
=h_{\mathrm{diff}}+H_*\sum_{k_\ell>M}A_\ell^2.
\]
The series tail tends to zero as $M\to\infty$.
Monotonicity in $M$ removes the countable exceptional set and proves
\eqref{eq:coefficientescape}.
The strong remainder criterion follows from
\eqref{eq:diffuseenergydefinition}.
\end{proof}

Formula \eqref{eq:coefficientescape} reduces the exclusion of diffuse
energy to tightness with respect to the transformed coefficient.
A positive coefficient moment provides that tightness. On a profile
of amplitude $A_\ell$, its contribution is proportional to
$A_\ell^{2-\theta(p-1)}$. The exponent
$\theta_c=2/(p-1)$ is therefore distinguished: every nonzero profile
then has the same positive cost, so a uniform moment bound also
controls their number.

\begin{corollary}\label{cor:coefficientmoments}
Under the assumptions of Theorem~\ref{thm:scalarprofiles}, suppose that,
for some $\theta>0$,
\begin{equation}\label{eq:coefficientmomentassumption}
C_\theta:=\sup_j\int_{\mathbb R}
K_j^{1+\theta}W_j^{p+1}\,d\tau<\infty.
\end{equation}
Then $h_{\mathrm{diff}}=0$ and
\begin{equation}\label{eq:amplitudemomentbound}
H_*k_*^\theta\sum_{\ell=1}^m
A_\ell^{2-\theta(p-1)}\le C_\theta.
\end{equation}
In particular, put $\theta_c=2/(p-1)$ and
$z_j=K_jW_j^{p-1}$. If
\[
C_{\theta_c}=\sup_j\int_{\mathbb R}
z_j^{(p+1)/(p-1)}\,d\tau<\infty,
\]
then
\begin{equation}\label{eq:finiteprofilemoment}
m\le\frac{C_{\theta_c}}{H_*k_*^{\theta_c}}<\infty,
\qquad
\left\|W_j-\sum_{\ell=1}^m
A_\ell W_*(\cdot-\tau_{j,\ell})\right\|_{\mathcal H}
\longrightarrow0.
\end{equation}
\end{corollary}

\begin{proof}[Proof of Corollary~\ref{cor:coefficientmoments}]
For $M>0$, \eqref{eq:coefficientmomentassumption} gives
\[
\sup_j\int_{\{K_j>M\}}K_jW_j^{p+1}\,d\tau
\le M^{-\theta}\sup_j\int_{\mathbb R}K_j^{1+\theta}W_j^{p+1}\,d\tau
\le C_\theta M^{-\theta}.
\]
Formula \eqref{eq:coefficientescape} therefore gives $h_{\mathrm{diff}}=0$.

Fix $L\le m$ and $R>0$. The first $L$ profile intervals are disjoint
for large $j$. Their local limits in
\eqref{eq:translatedcriticalprofiles} yield
\[
\begin{aligned}
C_\theta
&\ge\lim_{j\to\infty}\sum_{\ell=1}^L
\int_{\tau_{j,\ell}-R}^{\tau_{j,\ell}+R}
K_j^{1+\theta}W_j^{p+1}\,d\tau\\
&=k_*^{1+\theta}\sum_{\ell=1}^L
A_\ell^{2-\theta(p-1)}\int_{-R}^R W_*^{p+1}\,d\tau.
\end{aligned}
\]
First let $R\to\infty$, then $L\uparrow m$.
Since $k_*\int_{\mathbb R}W_*^{p+1}=H_*$, monotone convergence gives
\eqref{eq:amplitudemomentbound}.
At $\theta=\theta_c$, every term in the amplitude sum equals one. Hence
$mH_*k_*^{\theta_c}\le C_{\theta_c}$ and $m<\infty$.
The identity
\[
K_j^{1+\theta_c}W_j^{p+1}
=\bigl(K_jW_j^{p-1}\bigr)^{(p+1)/(p-1)}
\]
gives the stated phase criterion. Finally,
\eqref{eq:diffuseenergydefinition} with $L=m$ and $h_{\mathrm{diff}}=0$
proves the norm convergence in \eqref{eq:finiteprofilemoment}.
\end{proof}

The moment estimate is one-sided. A bound in
\eqref{eq:coefficientmomentassumption} does not identify the limit of the
moments themselves. Such an identification would require control of their
tails beyond the unweighted coefficient-energy tightness.

\begin{remark}\label{rem:compactnesslevels}
The condition $h_{\mathrm{diff}}=0$ means
\[
\lim_{L\uparrow m}\lim_{j\to\infty}\|R_j^L\|_{\mathcal H}=0.
\]
The number of profiles may still be infinite. The critical moment in
Corollary~\ref{cor:coefficientmoments} gives the finite strong decomposition
\eqref{eq:finiteprofilemoment}. Single-profile convergence requires the
exclusion of every further profile. In the peak normalization, $H_j\to H_*$
gives
\[
m=1,\qquad A_1=1,\qquad h_{\mathrm{diff}}=0,\qquad
\|W_j-W_*\|_{\mathcal H}\to0.
\]
For each fixed $j$, the positive lower bound on $\widehat c_j$ already gives
$K_j(\tau)\to\infty$ as $\tau\to\infty$.
The criterion \eqref{eq:coefficientescape} concerns energy escaping to
large coefficient values, not boundedness of the functions $K_j$.
\end{remark}

\subsection{Period insertion and logarithmic spreading}

We now turn from the description of possible limits to their
construction. Inserting complete periods into a critical autonomous
segment repeats a profile without increasing the maximal defect.
This first construction gives multiple profiles of equal limiting
amplitude and shows why a defect bound alone cannot bound $H$.
It will also supply the orbit segments used to prescribe unequal
amplitudes in the next subsection.

\begin{proposition}\label{prop:periodinsertion}
Let $n,p$ satisfy \eqref{eq:radialrange}.
For each $T\ge2$ in the base family of
Proposition~\ref{prop:contrastconstruction}, there are
$L_T,J_T,Z_T>0$ and pairs $(c_{T,k},U_{T,k})\in\mathcal A_{n,p}$,
$k=1,2,\ldots$, such that
\begin{equation}\label{eq:periodinsertiondata}
\begin{aligned}
D(c_{T,k})&=D(c_T),\\
\log\mathcal R(c_{T,k})&=\log\mathcal R(c_T)+d_*kL_T,\\
H(c_{T,k},U_{T,k})&=\frac{H_T+kJ_T}{Z_T}.
\end{aligned}
\end{equation}
The coefficients are nonincreasing and $C^1$ on $\mathbb R^n$.
Each solution has a positive fast tail.
If $E_T<0$ denotes the negative minimum of the base Pohozaev function
used in the construction, then
\begin{equation}\label{eq:periodlengthasymptotic}
L_T=\frac2N\log\frac1{|E_T|}+O(1),\qquad J_T\to H_*,\qquad Z_T\to1.
\end{equation}
For each fixed positive integer $k$,
\begin{equation}\label{eq:insertedenergylevels}
H(c_{T,k},U_{T,k})\longrightarrow(k+1)H_*,\qquad
\int_0^\infty r^{n-1}|V_{T,k}'-V_*'|^2\,dr\longrightarrow kH_*,
\end{equation}
where $V_{T,k}$ is any peak normalization.
In particular, the normalized energy is unbounded in every positive
defect window:
\begin{equation}\label{eq:unboundednormalizedenergy}
\sup_{\substack{(c,U)\in\mathcal A_{n,p}\\
d_*<D(c)\le d_*+\varepsilon}}H=\infty
\qquad(\varepsilon>0).
\end{equation}
There are near-threshold sequences with
$H_j\to\infty$ and $(D(c_j)-d_*)H_j\to0$.
For the diagonal system $p=q$, $\alpha=\beta=0$, these pairs realize
alternative {\rm(ii)} of Theorem~\ref{thm:sequencelimits}.
\end{proposition}
\begin{proof}
Normalize the base pair by
$V_T=U_T/U_T(1)$ and $\widehat c_T=c_TU_T(1)^{p-1}$.
Its maximum of $r^NV_T^2$ is one and its energy is $H_T$.
Set
\[
W(\tau)=e^{N\tau/2}V_T(e^\tau),\qquad
K(\tau)=e^{d_*\tau}\widehat c_T(e^\tau).
\]
The equation and the scalar Pohozaev function become
\begin{equation}\label{eq:periodmechanical}
\begin{gathered}
\ddot{W}-\frac{N^2}{4}W+KW^p=0,\qquad
\dot{P}=\frac{\dot{K}}{p+1}W^{p+1},\\
P=\frac{(\dot{W})^2}{2}-\frac{N^2}{8}W^2+\frac{K}{p+1}W^{p+1}.
\end{gathered}
\end{equation}
The phase construction gives
\[
P=\frac{Qy(N-y)}{p+1}(h_T-h_*).
\]
It is negative at some point with $s>T$, whereas
$P(\pm\infty)=0$. It therefore attains a negative minimum at
$\tau_0>T/a$. At that point $\dot{K}(\tau_0)=0$.
Moreover $\dot{W}(\tau_0)=(N/2-y(\tau_0))W(\tau_0)<0$.

Put $k_0=K(\tau_0)>0$ and consider the autonomous equation
\[
\ddot{w}-\frac{N^2}{4}w+k_0w^p=0,\qquad
w(0)=W(\tau_0),\quad \dot{w}(0)=\dot{W}(\tau_0).
\]
Its potential is
\[
F(w)=-\frac{N^2}{8}w^2+\frac{k_0}{p+1}w^{p+1}.
\]
On $(0,\infty)$, $F$ decreases to its unique negative minimum and then
increases to infinity. The energy $E=P(\tau_0)$ satisfies
\[
\min_{w>0}F(w)<E<0.
\]
The first inequality is strict because $\dot{w}(0)\ne0$.
Thus $F(w)=E$ has two positive simple roots $w_-<w_+$.
The orbit is positive and periodic, with period
\[
L_T=2\int_{w_-}^{w_+}\frac{dw}{\sqrt{2(E-F(w))}}<\infty.
\]
Since $E<0$, its energy identity also gives
$|\dot{w}|<Nw/2$.

Insert $k$ complete periods at $\tau_0$:
\[
\widetilde W_k(\tau)=
\begin{cases}
W(\tau),&\tau\le\tau_0,\\
w(\tau-\tau_0),&\tau_0\le\tau\le\tau_0+kL_T,\\
W(\tau-kL_T),&\tau\ge\tau_0+kL_T,
\end{cases}
\]
and
\[
\widetilde K_k(\tau)=
\begin{cases}
K(\tau),&\tau\le\tau_0,\\
k_0,&\tau_0\le\tau\le\tau_0+kL_T,\\
K(\tau-kL_T),&\tau\ge\tau_0+kL_T.
\end{cases}
\]
Periodicity matches $\widetilde W_k$ and its first derivative at both joins.
The equation matches its second derivative.
The equality $\dot{K}(\tau_0)=0$ makes $\widetilde K_k$ a $C^1$ function.
Define
\[
U_{T,k}(e^\tau)=e^{-N\tau/2}\widetilde W_k(\tau),\qquad
c_{T,k}(e^\tau)=e^{-d_*\tau}\widetilde K_k(\tau).
\]
Equation \eqref{eq:periodmechanical} gives
$-\Delta U_{T,k}=c_{T,k}U_{T,k}^p$.
They agree with the base pair near zero.
For $\tau\ge\tau_0+kL_T$,
\[
\begin{aligned}
U_{T,k}(e^\tau)&=e^{-NkL_T/2}V_T(e^{\tau-kL_T}),\\
c_{T,k}(e^\tau)&=e^{-d_*kL_T}\widehat c_T(e^{\tau-kL_T}).
\end{aligned}
\]
Hence the coefficient has a positive limit at infinity and the solution
has a positive fast tail. The pair belongs to $\mathcal A_{n,p}$.

On the inserted interval the signed coefficient defect is exactly $d_*$.
Outside it the signed defect equals that of the base pair, up to translation
in $\tau$. This proves $D(c_{T,k})=D(c_T)$ and monotonicity.
The origin value is unchanged, while the limit at infinity is multiplied by
$e^{-d_*kL_T}$. This proves the ratio formula in
\eqref{eq:periodinsertiondata}.

Set
\[
J_T=\int_0^{L_T}\left(\dot{w}-\frac N2w\right)^2\,d\tau>0,
\qquad Z_T=\max\left\{1,\max_{[0,L_T]}w^2\right\}.
\]
The change of variable $r=e^\tau$ gives
\[
\int_0^\infty r^{n-1}|U_{T,k}'|^2\,dr=H_T+kJ_T,
\qquad \max_{r>0}r^NU_{T,k}^2=Z_T.
\]
The energy formula follows.

We now compute its limit for fixed $k$.
Write $\varepsilon_T=D(c_T)-d_*$ and let $s_0$ correspond to $\tau_0$.
The scalar Pohozaev bound and $H_T\to H_*$ give
\[
0<-E\le\frac{\varepsilon_TH_T}{p+1}\longrightarrow0.
\]
For $s>0$, write $q(s)=q_T(-s)$, so $h_T=h_*-q$.
At $s_0>T$, the identity $d_T=d_*$ implies
\[
q'(s_0)=\frac{h_T(s_0)}{h_T(s_0)-1}
\tanh(s_0/2)q(s_0).
\]
Since $h_T\le h_-=1+2/N$, $s_0>T\ge2$, and
$n\tanh(1)/2>n/4=\kappa$, equation \eqref{eq:exponentialtaper} gives
\[
q(s_0)\le Ce^{-\kappa T}.
\]
For large $T$, this also gives $s_0<2T$.
The monotonicity of $q$ and the formula for $Q_T$ in the proof of
Theorem~\ref{thm:scalarminimum} yield
\[
\left|\log Q_T(s_0)-\log\bigl(\cosh(s_0/2)^{-2N/a}\bigr)\right|
\le CT e^{-\kappa T}.
\]
Using $K=h_Ty(N-y)Q_T^{-(p-1)/2}$, we obtain
\[
k_0=K(\tau_0)\longrightarrow k_*:=\frac{N(N+a)}4.
\]
The two turning points therefore satisfy $w_-\to0$ and $w_+\to1$.
In particular, $Z_T\to1$.
The periodic energy has the quadrature
\[
J_T=2\int_{w_-}^{w_+}
\left(\sqrt{2(E-F(w))}
+\frac{N^2w^2}{4\sqrt{2(E-F(w))}}\right)\,dw.
\]
These integrals converge to the corresponding homoclinic integral.
Here are the endpoint bounds.
For a fixed sufficiently small $b>0$, one has
$-F'(w)\ge cw$ on $(0,b)$, uniformly for large $T$.
Consequently
\[
E-F(w)\ge c(w^2-w_-^2),\qquad
\int_{w_-}^b\frac{w^2\,dw}{\sqrt{E-F(w)}}\le Cb^2.
\]
The integral of $\sqrt{E-F(w)}$ on this interval is also at most $Cb^2$.
Near $w_+$, the derivative $F'(w_+)$ stays bounded away from zero.
The square-root endpoint singularity is therefore uniformly integrable.
Convergence away from the endpoints is uniform.
The limiting orbit is
$W_*(\tau)=\cosh(a\tau/2)^{-N/a}$, and hence
\[
J_T\longrightarrow
\int_{\mathbb R}\left(\dot{W}_*-\frac N2W_*\right)^2\,d\tau=H_*.
\]
We also compute the period length. Uniformly for large $T$ and
$w_-\le w\le b$, with $b>0$ fixed and small,
\[
\frac1{\sqrt{2(E-F(w))}}
=\frac{2}{N\sqrt{w^2-w_-^2}}\bigl(1+O(w^{p-1})\bigr).
\]
This follows by subtracting $F(w_-)=E$ and using
\[
0\le\frac{w^{p+1}-w_-^{p+1}}{w^2-w_-^2}\le Cw^{p-1}.
\]
The integral of the error on $[w_-,b]$ is uniformly bounded.
The integral on $[b,w_+]$ is also bounded, by the simple upper turning point.
Since $w_-^2=8|E|/N^2\,(1+o(1))$,
\[
L_T=\frac4N\operatorname{arcosh}(b/w_-)+O(1)
=\frac2N\log\frac1{|E|}+O(1).
\]
Together with the limits for $J_T,Z_T$, this proves
\eqref{eq:periodlengthasymptotic}.
Formula \eqref{eq:periodinsertiondata} now gives the first limit in
\eqref{eq:insertedenergylevels}.
For fixed $k$, the normalized energy is bounded and
$\varepsilon_TH(c_{T,k},U_{T,k})\to0$.
The second limit follows from \eqref{eq:globalscalarstability}.

Given $\varepsilon>0$, choose $T$ so that $D(c_T)-d_*\le\varepsilon$.
Letting $k\to\infty$ proves \eqref{eq:unboundednormalizedenergy}.
Choose $T_j=j+2$ and $k_j=j$.
Then $H_j/j\to H_*$ and
\[
(D(c_{T_j,k_j})-d_*)H_j\le Cj e^{-\kappa j}\longrightarrow0.
\]
For the diagonal system,
\[
\mathcal D(c_{T_j,k_j},c_{T_j,k_j})-\gamma_0
=\frac{2(D(c_{T_j})-d_*)}{p+1}\longrightarrow0.
\]
Its normalized energy is the same $H_j$.
Thus \eqref{eq:uniformlogvanishing} holds for this sequence, while its
peak normalizations still converge locally to the nonzero critical profile.
\end{proof}

\subsection{Prescribed profile amplitudes and diffuse energy}

Period insertion fixes the repeated amplitude. To realize a prescribed
list, we must connect orbits with different amplitudes while keeping
the coefficient monotone and the extra defect small. The following
result gives the converse to the amplitude and energy description
in Theorem~\ref{thm:scalarprofiles}.

\begin{theorem}\label{thm:energyrealization}
Assume \eqref{eq:radialrange}. Let
$1=A_1\ge A_2\ge\cdots>0$ be a finite or countable list with
$\sum_\ell A_\ell^2<\infty$, and let $d\ge0$.
There is a near-threshold sequence in $\mathcal A_{n,p}$ whose nonzero
profile amplitudes in Theorem~\ref{thm:scalarprofiles} are exactly this list
and whose diffuse energy is $d$. Its coefficients can be chosen nonincreasing
and $C^1$, and each solution has a positive fast tail. In particular,
\[
H_j\longrightarrow H_*\sum_\ell A_\ell^2+d.
\]
The set of finite limits of normalized energies is exactly $[H_*,\infty)$.
For every $d>0$, one may arrange
\begin{equation}\label{eq:purediffuseexample}
\|W_j-W_*\|_{L^\infty(\mathbb R)}\longrightarrow0,
\qquad
\|W_j-W_*\|_{\mathcal H}^2\longrightarrow d,
\qquad H_j\longrightarrow H_*+d.
\end{equation}
\end{theorem}

The coefficients vary with the sequence and with the scale. The
coefficient $k_*A_\ell^{1-p}$ fixes the profile energy $A_\ell^2H_*$,
so these energies need not be integer multiples of $H_*$. No realization
result for a fixed coefficient, or prescription of arbitrary profile
centers, is asserted.

The next pulse changes the amplitude with vanishing energy cost.
It is used first for finite lists, before the final diagonal construction
produces countable lists and diffuse energy.

\begin{lemma}\label{lem:coefficientpulse}
Fix $0<\lambda<1$ and $0<k_-\le k_+<\infty$.
There are $L_\lambda>0$ and $E_0=E_0(n,p,\lambda,k_-,k_+)>0$ such that,
for $k\in[k_-,k_+]$ and $0<-E<E_0$,
there are a smooth positive coefficient $K$ on $[0,L_\lambda]$ and a positive solution
of
\[
\ddot{W}-\frac{N^2}{4}W+KW^p=0
\]
such that
\[
\begin{gathered}
K(0)=k,\qquad K(L_\lambda)=\lambda^{1-p}k,\qquad
0\le \dot{K}/K\le d_*,\\
P(0)=E,\qquad P(L_\lambda)=\lambda^2E,\qquad
\dot{W}(0),\dot{W}(L_\lambda)<0,\\
W(0)\asymp W(L_\lambda)\asymp |E|^{1/(p+1)}.
\end{gathered}
\]
Here $P$ is the mechanical energy in \eqref{eq:periodmechanical}.
All positive-order derivatives of $K$ vanish at the endpoints.
The pulse has energy
$\int_0^{L_\lambda}(\dot{W}-NW/2)^2\,d\tau=O(|E|^{2/(p+1)})$.
The comparison constants and the energy bound are uniform for
$k\in[k_-,k_+]$. They may depend on $n,p,\lambda,k_-,k_+$.
\end{lemma}
\begin{proof}
Use the fixed smooth cutoff $\vartheta$ from the scalar construction.
Put $b=(1-p)\log\lambda>0$ and choose
$L_\lambda\ge1+b\|\vartheta'\|_\infty/d_*$.
Set
\[
K(\tau)=k\exp\{b\vartheta(\tau/L_\lambda)\}.
\]
The coefficient has the required endpoint and derivative properties.
For a parameter $C>0$, prescribe
\[
\begin{aligned}
W(0)&=w_C=C|E|^{1/(p+1)},\\
\dot{W}(0)&=-\left(\frac{N^2}{4}w_C^2
-\frac{2k}{p+1}w_C^{p+1}+2E\right)^{1/2}.
\end{aligned}
\]
The radicand is positive for small $|E|$.
Uniformly for $C$ in a compact subset of $(0,\infty)$,
\[
\frac{\dot{W}(0)}{w_C}\longrightarrow-\frac N2,
\qquad
\frac{W}{w_C}\longrightarrow e^{-N\tau/2}
\quad\hbox{in }C^1([0,L_\lambda]).
\]
Indeed, the normalized equation has nonlinear coefficient
$K w_C^{p-1}\to0$. Continuous dependence for this initial-value problem
also gives positivity and negative derivative throughout the pulse.

Define
\[
B(k,\lambda)=\frac1{p+1}
\int_0^{L_\lambda}\dot{K}(\tau)e^{-N(p+1)\tau/2}\,d\tau>0.
\]
The identity $\dot{P}=\dot{K}W^{p+1}/(p+1)$ gives
\[
\frac{P(L_\lambda)}{|E|}
=-1+B(k,\lambda)C^{p+1}+o(1).
\]
The error is uniform on compact $C$ intervals and on $[k_-,k_+]$.
Set
\[
B_-:=\min_{k\in[k_-,k_+]}B(k,\lambda)>0,
\qquad B_+:=\max_{k\in[k_-,k_+]}B(k,\lambda)<\infty.
\]
Choose $0<C_-<C_+$ so that
\[
B_+C_-^{p+1}<\tfrac12(1-\lambda^2),\qquad
B_-C_+^{p+1}>2(1-\lambda^2).
\]
Choose $E_0>0$ so that the initial radicand is positive, the normalized
solution has negative derivative on the whole pulse, and the two energy
values bracket $\lambda^2E$, uniformly for $k\in[k_-,k_+]$.
Continuity in $C$ then gives $C\in(C_-,C_+)$ with
$P(L_\lambda)=\lambda^2E$.
Since $P$ is nondecreasing,
$E\le P\le\lambda^2E<0$ throughout the pulse.
The displayed $C^1$ convergence gives the endpoint size estimates and
the asserted energy bound.
\end{proof}

\begin{lemma}\label{lem:finiteamplitudes}
Let $1=A_1\ge\cdots\ge A_m>0$.
There are pairs $(c_{T},U_{T})\in\mathcal A_{n,p}$, for sufficiently large
$T$, with nonincreasing coefficients, positive fast tails, and the same
defect as the base family in Proposition~\ref{prop:contrastconstruction}.
Their peak-normalized logarithmic profiles satisfy
\begin{equation}\label{eq:finiteamplituderecovery}
\left\|W_T-\sum_{\ell=1}^m
A_\ell W_*(\cdot-t_{T,\ell})\right\|_{\mathcal H}\longrightarrow0,
\qquad
|t_{T,\ell}-t_{T,k}|\longrightarrow\infty\quad(\ell\ne k).
\end{equation}
One may label a maximizing center by $t_{T,1}=0$.
Consequently $H_T\to H_*\sum_{\ell=1}^m A_\ell^2$.
\end{lemma}
\begin{proof}
Use the base functions $W,K$, the cut $\tau_0$, the frozen coefficient
$k_0=K(\tau_0)$, and the periodic orbit from the proof of
Proposition~\ref{prop:periodinsertion}.
In this proof let $E_T=P(\tau_0)<0$ and $\zeta_T=W(\tau_0)$.
That proof gives $k_0\to k_*$, $E_T\to0$, and $\zeta_T\to0$.
It also gives
\begin{equation}\label{eq:pulsecutseparation}
\frac{|E_T|}{\zeta_T^{p+1}}
=\frac{k_0}{p+1}
\frac{h_*-h_T(s_0)}{h_T(s_0)}\longrightarrow0.
\end{equation}
The smaller turning point is comparable with $|E_T|^{1/2}$.
Since $p>1$, \eqref{eq:pulsecutseparation} gives the scale separation
\[
\frac{|E_T|^{1/2}}{|E_T|^{1/(p+1)}}\longrightarrow0,
\qquad
\frac{|E_T|^{1/(p+1)}}{\zeta_T}\longrightarrow0.
\]
The pulse therefore lies below the descending cut and above the lower
turning point. The same relations hold after any fixed positive amplitude
rescaling.

Fix the finite amplitude list before choosing $T$.
For each strict drop $A\to B$, fix $\lambda=B/A$ and a compact interval
of positive coefficients containing $A^{1-p}k_0$ for all large $T$.
Lemma~\ref{lem:coefficientpulse} then fixes its pulse length and energy
threshold. Since there are finitely many drops, one common sufficiently
large $T$ satisfies all thresholds and scale separations.

Retain the base pair up to $\tau_0$.
At a descending cut of amplitude $A$, the data are
\[
(w,\dot{w})=A\bigl(W(\tau_0),\dot{W}(\tau_0)\bigr),\qquad
K=A^{1-p}k_0,\qquad P=A^2E_T.
\]
To repeat the same amplitude, insert a full period of the scaled orbit.
To pass from $A$ to $B<A$, put $\lambda=B/A$.
Follow the descending periodic orbit to the initial height supplied by
Lemma~\ref{lem:coefficientpulse}, and apply that pulse with
$k=A^{1-p}k_0$ and $E=A^2E_T$.
Equation \eqref{eq:pulsecutseparation} ensures that this height occurs
strictly after the cut and before the lower turning point.
The pulse ends on the descending branch of the orbit with
\[
K=B^{1-p}k_0,\qquad P=B^2E_T.
\]
This is exactly the $B$-multiple of the base periodic orbit.
Its endpoint height is smaller than $B\zeta_T$ for large $T$.
Continue through the lower turning point and the next maximum, and stop
at the next descending occurrence of
$B(W(\tau_0),\dot{W}(\tau_0))$.
Repeat this procedure for the finite amplitude list.
At the final cut, attach the translated $A_m$-multiple of the base tail,
with coefficient $A_m^{1-p}K$ in logarithmic variables.

Writing $[f]$ for the difference of the right and left traces at a join,
the construction gives
\[
[W]=[\dot{W}]=[\ddot{W}]=0,\qquad [K]=[\dot{K}]=0.
\]
The equation gives the match of $\ddot{W}$. Flat pulse endpoints and
$\dot{K}(\tau_0)=0$ give the match of $\dot{K}$. The resulting radial coefficient is $C^1$ and the solution
is $C^2$. They agree with the base pair near zero.
The final tail is a dilation and constant rescaling of the base tail,
so its coefficient limit and fast-decay constant are positive.
On a pulse the signed radial defect is
$d_*-\dot{K}/K\in[0,d_*]$; on a frozen orbit it is $d_*$.
On the retained base pieces it is the base defect.
The supremum of the base defect occurs in its negative-$s$ half:
at paired points $-s,s$, the algebraic term in
\eqref{eq:phasecontrol} is equal, while $(h-1)/h$ is larger at $-s$.
The whole negative-$s$ half was retained.
Thus the defect is unchanged and the coefficient remains nonincreasing.
Its endpoint limits give positive upper and lower bounds.

We compute the energy before peak normalization.
The retained prefix has energy tending to $H_*$.
Each new excursion of amplitude $A_\ell$ has energy tending to
$A_\ell^2H_*$, by the periodic quadrature in the proof of
Proposition~\ref{prop:periodinsertion}.
The omitted parts of these excursions lie below heights tending to zero.
The bound
\[
\int_{w_-}^{b}\frac{w^2\,dw}{\sqrt{E-F(w)}}\le Cb^2
\]
from that proof, together with the boundary term from $w\dot{w}$, shows that
their energies tend to zero.
Each pulse has vanishing energy by Lemma~\ref{lem:coefficientpulse}.
The final tail has vanishing energy by the uniform tail estimate for the
base family. Hence the total energy tends to
$H_*\sum_{\ell=1}^m A_\ell^2$.

The peak heights tend to $A_1,\ldots,A_m$ and the global squared peak tends
to one. At every peak, autonomous initial-value dependence gives the local
limit $A_\ell W_*$.
Between successive peaks the profile passes through a height tending to
zero. Since $|\dot{W}|\le NW/2$ on each negative-energy piece, their separations
tend to infinity. The same separation from the first peak follows from
$\tau_0>T/a$.
After peak normalization, label a maximizing peak first; if several
amplitudes equal one, relabel those amplitudes.
Local weak convergence, separation, and expansion of the $\mathcal H$
norm now give \eqref{eq:finiteamplituderecovery}.
\end{proof}

\begin{proof}[Proof of Theorem~\ref{thm:energyrealization}]
For a finite amplitude list and $d=0$, apply
Lemma~\ref{lem:finiteamplitudes}.
For the general case, let $m_j$ be the smaller of $j$ and the number of
prescribed amplitudes. If $d>0$, choose integers $k_j\to\infty$ so large that
\[
B_j=\left(\frac{d}{k_jH_*}\right)^{1/2}
\le\min\{j^{-1},A_{m_j}\}.
\]
Append $k_j$ copies of $B_j$ to the finite list
$A_1,\ldots,A_{m_j}$. If $d=0$, no extra amplitudes are appended.
For each $j$, apply Lemma~\ref{lem:finiteamplitudes} to this finite list.
Choose its parameter sufficiently large that the defect excess is at most
$j^{-1}$, the $\mathcal H$ error in
\eqref{eq:finiteamplituderecovery} is at most $j^{-1}$, all center
separations exceed $j$, and the total energy differs by at most $j^{-1}$
from
\[
H_*\sum_{\ell=1}^{m_j}A_\ell^2+d.
\]
For clarity, the order of choices is
\[
(m_j,k_j,B_j)\ ;\quad
\{\lambda_{j,i},[k^-_{j,i},k^+_{j,i}],E_{0,j,i}\}_{i}\ ;\quad T_j.
\]
The middle list contains the finitely many strict amplitude drops;
equal amplitudes use period insertion. Choose $T_j$ large enough that
each incoming energy satisfies
$A_{j,i}^2|E_{T_j}|<E_{0,j,i}$, that the scale separations hold at every
pulse, and that the preceding approximation bounds hold.
One may also require $T_j>T_{j-1}$. The constants need not be uniform as
$j\to\infty$: all choices are made for a fixed finite list before $T_j$
is chosen.

For centers separated by at least $j$, the exponential tails of $W_*$ give
\[
\sup_\tau\sum_\ell W_*(\tau-t_{j,\ell})\le C
\]
with $C$ independent of the number of centers and $j\ge1$.
The sum of the added small profiles has uniform norm at most $CB_j\to0$.
The profiles from the prescribed tail $\ell>L$ have uniform norm at most
$CA_{L+1}$, which tends to zero when the list is infinite.
The $\mathcal H$ approximation also tends uniformly to zero by the
one-dimensional Sobolev inequality.
Thus the nonzero translated profiles have exactly the prescribed amplitudes.
Their squared energies sum to $H_*\sum_\ell A_\ell^2$.
The limiting total energy is that sum plus $d$, so
\eqref{eq:diffuseenergydefinition} gives $h_{\mathrm{diff}}=d$.

Taking only $A_1=1$ proves \eqref{eq:purediffuseexample}.
For any $h\ge H_*$, take this single amplitude and $d=h-H_*$.
Conversely, \eqref{eq:minimumenergyliminf} excludes smaller finite limits.
All coefficients and tails have the properties established in
Lemma~\ref{lem:finiteamplitudes}.
\end{proof}

The construction separates two obstructions to single-profile
convergence. Choose $A_1=A_2=1$ and $d=0$ in
Theorem~\ref{thm:energyrealization} gives
\[
h_{\mathrm{diff}}=0,\qquad H_j\to2H_*,\qquad
\|W_j-W_*\|_{\mathcal H}^2\to H_*.
\]
The last limit also follows from \eqref{eq:globalscalarstability}.
Thus coefficient-energy tightness alone does not give a single profile.
Formula \eqref{eq:purediffuseexample} gives the other obstruction:
uniform convergence to $W_*$ can coexist with positive diffuse energy.
Together with the stability and profile theorems, these conclusions
prove the scalar overview in Theorem~\ref{thm:scalaroverview}.

\section{Sharp defects for weighted Newton systems}\label{sec:systems}

The scalar obstruction extends through a weighted Pohozaev balance,
but sharpness requires simultaneous control of two component slopes.
We first exclude the endpoint and construct near-minimizers by gluing
critical system orbits to admissible tails. Energy normalization then
gives the fixed-ratio gap. Finally, the compactness available for the
coefficients leads to a BV relaxation of that constrained problem.
The homogeneous Liouville problem has a separate history: Souplet's
four-dimensional theorem~\cite{Souplet2009} concerns the unweighted
system, and the recent preprints~\cite{MeiWei,LiLuZhang2026} concern
homogeneous fractional subcritical problems.

Throughout this section, assume \eqref{eq:systemradialrange} and write
\[
N=n-2,\qquad A=2+\alpha,\qquad B=2+\beta.
\]
Thus $A,B>0$, and the exponents in \eqref{eq:systemradialrange} satisfy
\begin{equation}\label{eq:systemscaling}
pt=s+A,\qquad qs=t+B,\qquad
\gamma_0=\frac{pq-1}{(p+1)(q+1)}(s+t-N).
\end{equation}
In particular, the subcritical assumption is equivalent to $s+t>N$.
We retain the class $\mathcal B$ and the defects and ratios defined in
Section~1. Thus the coefficient factors are $C^1$ off zero and double
bounded, and the solutions are positive and continuous at zero.
Their absolute logarithmic derivatives are bounded for each quadruple,
without a common bound on the whole class. Lemma~\ref{lem:newton}
gives the Newton potential representation.

\subsection{The strict endpoint obstruction}

The endpoint argument needs vanishing Pohozaev boundary terms.
The following flux and decay bounds provide them using only the
positive coefficient lower bounds.

\begin{lemma}\label{lem:systemdecay}
For every quadruple in $\mathcal B$,
\begin{equation}\label{eq:systemflux}
-r^{n-1}u'(r)=\int_0^r \xi^{n-1+\alpha}a_0(\xi)v(\xi)^p\,d\xi,
\qquad
-r^{n-1}v'(r)=\int_0^r \xi^{n-1+\beta}b_0(\xi)u(\xi)^q\,d\xi.
\end{equation}
Both components are strictly decreasing, and
\begin{equation}\label{eq:systemdecay}
u(r)\le Cr^{-s},\quad v(r)\le Cr^{-t},\quad
|u'(r)|\le Cr^{-s-1},\quad |v'(r)|\le Cr^{-t-1}.
\end{equation}
The constants can be chosen uniformly when the two coefficient lower bounds are fixed positive numbers. The same assertions hold for bounded measurable positive coefficients and solutions of the flux equations that are continuous and positive at zero.
\end{lemma}
\begin{proof}
Integration of the radial equations gives \eqref{eq:systemflux}. The integration constants at zero vanish: the sources are integrable there because $\alpha,\beta>-2$, and a nonzero constant flux would produce an $r^{2-n}$ singularity, contradicting continuity. The fluxes are positive and nondecreasing. Positivity of $u$ therefore gives
\[
u(r)\ge\int_r^\infty[-\xi^{n-1}u'(\xi)]\xi^{1-n}\,d\xi
\ge\frac{-ru'(r)}N.
\]
The same argument applies to $v$. Since both components decrease,
\[
u(r)\ge\frac{a_-}{N(n+\alpha)}r^A v(r)^p,
\qquad
v(r)\ge\frac{b_-}{N(n+\beta)}r^B u(r)^q,
\]
where $a_-=\inf a_0>0$ and $b_-=\inf b_0>0$. Substitution of one inequality into the other gives \eqref{eq:systemdecay} for $u,v$, using $pq>1$ and the definitions of $s,t$. The derivative estimates follow from $-ru'\le Nu$ and $-rv'\le Nv$. Every step uses only the flux equations and coefficient lower bounds, so the last assertion follows as well.
\end{proof}

\begin{proposition}\label{prop:systemendpoint}
Under \eqref{eq:systemradialrange}, every quadruple in $\mathcal B$ satisfies \eqref{eq:systemstrict}.
\end{proposition}
\begin{proof}
Put $d_a=D(a_0)$ and $d_b=D(b_0)$, and suppose, to the contrary, that
\[
\Gamma=\gamma_0-\frac{d_a}{p+1}-\frac{d_b}{q+1}\ge0.
\]
Set
\[
\ell=\frac{n+\alpha-d_a}{p+1},\qquad m=N-\ell,
\]
and define the radial Pohozaev function
\begin{equation}\label{eq:systempohozaev}
\begin{split}
P(r)={}&r^n\left[u'v'+\frac{r^\alpha a_0v^{p+1}}{p+1}
+\frac{r^\beta b_0u^{q+1}}{q+1}\right]\\
&+r^{n-1}(muv'+\ell vu').
\end{split}
\end{equation}
Differentiating and using $\ell+m=N$ gives
\begin{equation}\label{eq:systempohozaevderivative}
\begin{split}
P'(r)={}&\frac{r^{n-1+\alpha}a_0v^{p+1}}{p+1}
\left(d_a+\frac{ra_0'}{a_0}\right)\\
&+\frac{r^{n-1+\beta}b_0u^{q+1}}{q+1}
\left(d_b+\frac{rb_0'}{b_0}\right)
+\Gamma r^{n-1+\beta}b_0u^{q+1}\ge0.
\end{split}
\end{equation}
Near zero, the flux equations give $u'=O(r^{A-1})$ and $v'=O(r^{B-1})$, so $P(r)\to0$. At infinity, Lemma~\ref{lem:systemdecay} and \eqref{eq:systemscaling} imply
\[
|P(r)|\le Cr^{N-s-t}\longrightarrow0.
\]
Hence $P\equiv0$. If $\Gamma>0$, the final term in \eqref{eq:systempohozaevderivative} gives an immediate contradiction. If $\Gamma=0$, the two other nonnegative terms vanish separately. Thus
\[
a_0(r)=c_a r^{-d_a},\qquad b_0(r)=c_b r^{-d_b}.
\]
Since $d_a/(p+1)+d_b/(q+1)=\gamma_0>0$, at least one exponent is positive, contradicting the positive upper and lower coefficient bounds.
\end{proof}

\subsection{Near-critical constructions}

To prove sharpness, we begin with a homogeneous critical system and
restore bounded coefficient factors at both endpoints. The two
component slopes can approach different limiting equilibria, so the
tail construction must respect both defect budgets simultaneously.
The quantitative target is stated first; the critical-orbit estimates
below provide the gluing data.
Fix $d_a,d_b$ such that
\begin{equation}\label{eq:systembudgets}
0<d_a<A,\qquad 0<d_b<B,\qquad
\frac{d_a}{p+1}+\frac{d_b}{q+1}=\gamma_0.
\end{equation}
This is possible because $0<\gamma_0<A/(p+1)+B/(q+1)$.
Choose them off the two points
$pN=N+A-d_a$ and $qN=N+B-d_b$.
All constants below may depend on this fixed choice.
The restrictions $d_a<A$ and $d_b<B$ are imposed only for this construction.
They are not assumed for the limiting defects in
Theorem~\ref{thm:sequencelimits}.

\begin{lemma}\label{lem:systemconstruction}
Assume \eqref{eq:systemradialrange}. There are $C,T_0>0$ such that,
for every $T\ge T_0$, a quadruple
$(a_{0,T},b_{0,T},u_T,v_T)\in\mathcal B$ satisfies
\begin{equation}\label{eq:systemquantitative}
\gamma_0<\mathcal D(a_{0,T},b_{0,T})
\le\gamma_0+\frac CT,\qquad
\log\mathcal R(a_{0,T},b_{0,T})\le CT.
\end{equation}
All profiles are smooth for $r>0$. For all sufficiently large $r$,
the coefficients are constant and $u_T=C_u r^{-s}$, $v_T=C_v r^{-t}$.
\end{lemma}

\begin{lemma}\label{lem:criticalorbitestimates}
The system
\begin{equation*}
-\Delta u_*=r^{\alpha-d_a}v_*^p,\qquad
-\Delta v_*=r^{\beta-d_b}u_*^q
\end{equation*}
has a positive radial ground state. Put $\tau=\log r$ and
$x_*=-\frac{e^\tau u_*'(e^\tau)}{u_*(e^\tau)}$, $y_*=-\frac{e^\tau v_*'(e^\tau)}{v_*(e^\tau)}$.
At $-\infty$, with $a=A-d_a>0$ and $b=B-d_b>0$,
\[
x_*=c_a e^{a\tau}(1+O(e^{\eta\tau})),\qquad
y_*=c_b e^{b\tau}(1+O(e^{\eta\tau})),\qquad c_a,c_b>0,
\]
and
\[
\left|\frac{\dot{x}_*}{x_*}-a\right|
+\left|\frac{d}{d\tau}\left(\frac{\dot{x}_*}{x_*}\right)\right|
+\left|\frac{\dot{y}_*}{y_*}-b\right|
+\left|\frac{d}{d\tau}\left(\frac{\dot{y}_*}{y_*}\right)\right|
\le Ce^{\eta\tau}
\]
for some $\eta>0$.
At $+\infty$, the endpoint $(X,Y)$ is one of
\[
(N,N),\qquad (pN-A+d_a,N),\qquad (N,qN-B+d_b).
\]
The second case occurs when $pN<N+A-d_a$; the third occurs when
$qN<N+B-d_b$. Otherwise both components are fast.
Every slow slope belongs to $(0,N)$, and its difference from the limit,
together with its first two derivatives, is $O(e^{-\eta\tau})$.
For a fast first component, set $\delta_*=N-x_*$ and
$\lambda_a=pY-A+d_a-N>0$. Then
\begin{equation}\label{eq:systemfastasymptotic}
\begin{gathered}
\delta_*=c e^{-\lambda_a\tau}(1+O(e^{-\eta\tau})),\qquad c>0,\\
\left|\frac{\dot{\delta}_*}{\delta_*}+\lambda_a\right|
+\left|\frac{d}{d\tau}\left(\frac{\dot{\delta}_*}{\delta_*}\right)\right|
\le Ce^{-\eta\tau}.
\end{gathered}
\end{equation}
The second component satisfies the analogous estimate.
\end{lemma}
\begin{proof}
The critical identity in \eqref{eq:systembudgets} gives
\[
\frac{n+\alpha-d_a}{p+1}
+\frac{n+\beta-d_b}{q+1}=N.
\]
Existence and the three nonresonant endpoints follow from
\cite[Theorem~1.4 and Section~6]{BG}. In the present notation the phase
variables are
\[
x_*,\quad y_*,\quad
Z_*=\frac{r^{A-d_a}v_*^p}{u_*x_*},\quad
W_*=\frac{r^{B-d_b}u_*^q}{v_*y_*}.
\]
They satisfy the polynomial system
\[
\begin{aligned}
\dot{x}_*&=x_*(x_*-N+Z_*),&
\dot{Z}_*&=Z_*(N+A-d_a-py_*-Z_*),\\
\dot{y}_*&=y_*(y_*-N+W_*),&
\dot{W}_*&=W_*(N+B-d_b-qx_*-W_*).
\end{aligned}
\]
The limiting equilibria are hyperbolic. At the fast--fast endpoint their
eigenvalues are $N,N,-\lambda_a,-\lambda_b$.
At the endpoint $(L,N)$ they are
$N,-\lambda_b,L,-(N-L)$; the other mixed endpoint is symmetric.
All these numbers are nonzero. Convergence to a hyperbolic equilibrium
is exponential along its stable manifold. The polynomial equations then
give exponential convergence of the first two derivatives as well.
Integrating the slope errors gives
$u_*(r)=C_u r^{-X}(1+O(r^{-\eta}))$ and
$v_*(r)=C_v r^{-Y}(1+O(r^{-\eta}))$ at infinity.
Decreasing $\eta$ absorbs any polynomial factor in the exponential estimates.

For a fast first component, the radial Newton formula gives the exact identity
\begin{equation}\label{eq:fastdeficittail}
N-x_*(\log r)=
\frac{\displaystyle\int_r^\infty
\xi^{1+\alpha-d_a}v_*(\xi)^p\,d\xi}{u_*(r)}.
\end{equation}
Since $A-d_a-pY=-N-\lambda_a$, the numerator is
\[
\frac{C_v^p}{N+\lambda_a}\,
r^{-N-\lambda_a}(1+O(r^{-\eta})).
\]
This proves the first estimate in \eqref{eq:systemfastasymptotic}.
Write the numerator in \eqref{eq:fastdeficittail} as $J(r)$ and put
$T(r)=r^{A-d_a}v_*(r)^p/J(r)$. Then
\[
\frac{\dot{\delta}_*}{\delta_*}=x_*-T,\qquad
rT'(r)=T(A-d_a-py_*+T).
\]
Here $T=N+\lambda_a+O(e^{-\eta\tau})$.
These identities prove the two remaining estimates.

At zero, the flux equations give
$u_*(r)=u_*(0)+O(r^a)$ and $v_*(r)=v_*(0)+O(r^b)$.
Consequently
\[
x_*=\frac{v_*(0)^p}{(n+\alpha-d_a)u_*(0)}
r^a(1+O(r^{\min(a,b)})),
\]
and the analogous formula holds for $y_*$.
The flux equations, or the displayed polynomial system, give the stated
logarithmic derivative estimates.
\end{proof}

\begin{proof}[Proof of Lemma~\ref{lem:systemconstruction}]
For slopes $0<x,y<N$, set
\[
z=\dot{x}+x(N-x),\qquad w=\dot{y}+y(N-y).
\]
When $z,w>0$, integrate $\frac{d}{d\tau}u(e^\tau)=-xu(e^\tau)$, $\frac{d}{d\tau}v(e^\tau)=-yv(e^\tau)$ and define
\begin{equation}\label{eq:systeminverse}
a_0(e^\tau)=e^{-A\tau}\frac{u(e^\tau)}{v(e^\tau)^p}z,\qquad
b_0(e^\tau)=e^{-B\tau}\frac{v(e^\tau)}{u(e^\tau)^q}w.
\end{equation}
Then the radial equations hold and
\begin{equation}\label{eq:systemslopeidentity}
\mathfrak a:=-\frac{e^\tau a_0'(e^\tau)}{a_0(e^\tau)}
=A+x-py-\frac{\dot{z}}z,\qquad
\mathfrak b:=-\frac{e^\tau b_0'(e^\tau)}{b_0(e^\tau)}
=B+y-qx-\frac{\dot{w}}w.
\end{equation}
On the critical orbit these expressions equal $d_a,d_b$.

Fix $\vartheta\in C^\infty(\mathbb R)$, nondecreasing, equal to zero
on $(-\infty,0]$ and one on $[1,\infty)$.
Write $C_\vartheta=1+\|\vartheta'\|_\infty+\|\vartheta''\|_\infty$.
We first fix the small constants used below, then a sufficiently large
$c_1$, and finally $T_0$. None depends on $T$.
Retain the critical orbit on $[-c_1T,c_1T]$.

\smallskip\noindent\emph{The origin.}
Let $L=-c_1T$ and $k_*=\dot{x}_*/x_*$. On $[L-T,L]$ put
\[
k(\tau)=A+
\vartheta\left(\frac{\tau-L+T}{T}\right)(k_*(\tau)-A).
\]
Set $k=A$ to the left and solve $\dot{x}=kx$ backwards with $x(L)=x_*(L)$.
Use the same construction for $y$.
Lemma~\ref{lem:criticalorbitestimates} gives
\[
k\ge A-d_a-Ce^{-\eta c_1T}>0,\qquad
|\dot{k}|\le C/T,\qquad 0<x(\tau)\le Ce^{-\eta c_1T}\quad(\tau\le L).
\]
All jets match at $L$, because $\vartheta$ is flat at its endpoints.
The identities
\[
z=x(N-x+k),\qquad
\mathfrak a=A-k+x-py+\frac{kx-\dot{k}}{N-x+k}
\]
give $z>0$ and $\mathfrak a\le d_a+C/T$.
The corresponding statements hold for $w,\mathfrak b$.
To the left of $L-T$, $x=c_a e^{A\tau}$ and $y=c_b e^{B\tau}$.
Thus
\[
u=u_0\exp(-c_a r^A/A),\qquad
v=v_0\exp(-c_b r^B/B).
\]
The coefficients in \eqref{eq:systeminverse} have positive finite limits at zero.

\smallskip\noindent\emph{Departure from a fast endpoint.}
Let $R=c_1T$. For a fast first component write $x=N-\delta$.
For $h=\dot{\delta}/\delta$ one has
\begin{equation}\label{eq:systemfastexit}
z=\delta(N-\delta-h),\qquad
\mathfrak a=A+x-py-h+\frac{\delta h+\dot{h}}{N-\delta-h}.
\end{equation}
Choose $0<\varrho\le\min\{1,N/8\}$.
In a mixed case impose the further bound specified below.
On $[R,R+T]$ prescribe
\[
h=(1-\vartheta((\tau-R)/T))\frac{\dot{\delta}_*}{\delta_*}
+\vartheta((\tau-R)/T)\varrho,
\qquad \delta(R)=\delta_*(R).
\]
Then
\[
-\lambda_a-Ce^{-\eta R}\le h\le\varrho,\qquad
|\dot{h}|\le C/T.
\]
Choose $c_1$ so that $c_1\lambda_a>\varrho+1$ for every fast
component. Formula \eqref{eq:systemfastasymptotic} and
$-\lambda_a-1\le h\le\varrho$ then give
\[
e^{-C_2T}\le\delta(\tau)\le Ce^{-c_2T}
\quad(R\le\tau\le R+T)
\]
with fixed $C_2,c_2>0$. In particular $N-\delta-h\ge N/2$.
Formula \eqref{eq:systemfastexit} gives
$\mathfrak a\le d_a+C/T$ while the other slope stays exponentially
close to its endpoint.
For a slow component with limit $L_0$, put on this same interval
\[
x=(1-\vartheta((\tau-R)/T))x_*+\vartheta((\tau-R)/T)L_0.
\]
Its first two derivative errors are $O(e^{-\eta R})$.
Thus $z\ge L_0(N-L_0)/2$ and its defect is at most $d_a+C/T$.
Apply these formulas to the appropriate component or to both fast components.
At $R+T$ every fast $h$ equals $\varrho$, and every slow slope is constant.

\smallskip\noindent\emph{A fixed interior endpoint.}
Suppose first that both slopes are fast. Their defect margins are
\[
\lambda_a=d_a-(A+N-pN)>0,\qquad
\lambda_b=d_b-(B+N-qN)>0.
\]
Choose
\[
0<\Delta\le\min\left\{\frac N8,
\frac{\min(\lambda_a,\lambda_b)}{4(1+p+q)}\right\}.
\]
Each deficit will increase to exactly $\Delta$ and then remain constant.
Here is an explicit stopping rule. Put
\[
L_0=C_\vartheta/\varrho,\qquad
K_0=\varrho L_0\int_0^1(1-\vartheta(t))\,dt.
\]
If the initial deficit at time $S$ is $\delta_S$, keep $h=\varrho$ until
\[
G=S+\varrho^{-1}\log(\Delta e^{-K_0}/\delta_S).
\]
On $[G,G+L_0]$ set
$h=\varrho(1-\vartheta((\tau-G)/L_0))$, and set $h=0$ afterwards.
For large $T$, $G>S$. This rule gives
\[
0<\delta\le\Delta,\qquad
\delta(G+L_0)=\Delta,\qquad
0\le h\le\varrho,\qquad -\varrho^2\le \dot{h}\le0
\]
on the stopping interval. Here $N-\delta-h\ge3N/4$, and $\dot{h}\le0$. Therefore
\[
-h+\frac{\delta h+\dot{h}}{N-\delta-h}
\le-h+\frac{h}{6}\le0.
\]
Both defects therefore stay below their budgets. The two stopping times
may differ. Freeze the first completed component until the second completes.
The same bounds hold during this waiting interval.
The final slopes are $(x_0,y_0)=(N-\Delta,N-\Delta)$.

Suppose next that $(X,Y)=(L,N)$, with $0<L<N$.
Set $J=p+1$. At $S=R+T$, the first slope is $L$ and
the second is $N-\delta$, with $\dot{\delta}/\delta=\varrho$.
On $[S,S+1]$ put
\[
y=N-\delta,\qquad x=L-J\vartheta(\tau-S)\delta.
\]
All matching errors through the second derivative are $O(e^{-c_2T})$.
Thus $z,w>0$ and both defect errors are at most $C/T$ there.
After $S+1$ use the exact path
\[
y=N-\delta,\qquad x=L-J\delta.
\]
Along this path,
\[
\begin{aligned}
A+x-py&=d_a-(J-p)\delta,\\
z&=x(N-x)-Jh\delta,\\
\dot{z}&=-J\delta\{\dot{h}+h^2+h(N-2x)\}.
\end{aligned}
\]
Impose
\[
0<\varrho\le\min\left\{1,\frac N8,\frac{N-L}{2},
\frac{L(N-L)}{8J(N+2)}\right\},
\]
and then choose
\[
0<\Delta\le\min\left\{\frac N8,\frac{L}{2J},
\frac{\lambda_b}{4(qJ+1)}\right\},\qquad
\lambda_b=d_b-(B+N-qL)>0.
\]
For $0<\delta\le\Delta$, $0\le h\le\varrho$ and
$|\dot{h}|\le\varrho^2$, one has
\[
\frac L2\le x\le L,\qquad
x(N-x)\ge\frac{L(N-L)}2,\qquad
Jh\delta\le\frac{L(N-L)}4.
\]
Consequently
\[
z\ge\frac{L(N-L)}4,\qquad
\frac{|\dot{z}|}{z}
\le\frac{4J(2\varrho^2+N\varrho)}{L(N-L)}\delta
\le\frac\delta2.
\]
The stopping rule above, now starting at $S+1$, gives
\[
\mathfrak a\le d_a-\frac{J-p}{2}\delta,\qquad
B+y-qx=d_b-\lambda_b+(qJ-1)\delta.
\]
Formula \eqref{eq:systemfastexit} for the second component controls its defect.
The final slopes are $(x_0,y_0)=(L-J\Delta,N-\Delta)$.
The case $(N,L)$ follows by interchanging the components.

In all cases $(x_0,y_0)$ is fixed, belongs to $(0,N)^2$, and satisfies
\[
A+x_0-py_0\le d_a,\qquad B+y_0-qx_0\le d_b.
\]
The constants are chosen in the order
$(d_a,d_b)$, $\varrho$, $\Delta$, $c_1$, $T_0$.
The departure, growth, stopping, and waiting intervals have total length
$O(T)$. Indeed their initial deficits lie between $e^{-C_2T}$ and
$Ce^{-c_2T}$, and the displayed formula for $G$ is $O(T)$.

\smallskip\noindent\emph{The power tail.}
Let $S_0$ be the last stopping time. On $[S_0,S_0+T]$ set
\[
(x,y)=(x_0,y_0)+
\vartheta((\tau-S_0)/T)\bigl((s,t)-(x_0,y_0)\bigr).
\]
The segment lies in a fixed compact subset of $(0,N)^2$.
Its first two derivatives are bounded by $C/T$ and $C/T^2$.
Thus $z,w\ge c>0$ and $|\dot{z}/z|+|\dot{w}/w|\le C/T$.
The two affine terms in \eqref{eq:systemslopeidentity} are bounded by
$d_a,d_b$, since their values at $(s,t)$ are zero.
For $\tau\ge S_0+T$, set $(x,y)=(s,t)$.
Then
\[
u=C_ur^{-s},\qquad v=C_vr^{-t},\qquad
a_0=s(N-s)C_u/C_v^p,\qquad b_0=t(N-t)C_v/C_u^q.
\]
The coefficients are positive constants on this tail.

All joins are smooth. The origin limits and the final constants imply
positive upper and lower coefficient bounds for each fixed $T$.
The defects satisfy
$\mathfrak a\le d_a+C/T$, $\mathfrak b\le d_b+C/T$.
Their absolute values are uniformly bounded on all transition intervals.
The leftmost tail contributes at most
$C\int_{-\infty}^{L-T}(x+y)\,d\tau\le C$.
The rightmost tail contributes zero. Therefore
\[
\log\mathcal R(a_0,b_0)
\le\int_{\mathbb R}(|\mathfrak a|+|\mathfrak b|)\,d\tau\le CT.
\]
Equations \eqref{eq:systembudgets} and \eqref{eq:systemslopeidentity}
give the upper defect estimate. Proposition~\ref{prop:systemendpoint}
gives strictness. This proves \eqref{eq:systemquantitative}.
\end{proof}

\begin{proof}[Threshold in Theorem~\ref{thm:systemthreshold}]
Proposition~\ref{prop:systemendpoint} gives the strict lower bound.
Lemma~\ref{lem:systemconstruction} gives approximating quadruples with
defect tending to $\gamma_0$ and the stated power tails.
\end{proof}

\subsection{Energy balance and coefficient normalization}

The construction proves sharpness without controlling a prescribed
ratio. For the constrained problem we need an identity that measures
the defect excess and a normalization compatible with coefficient
compactness. The next two lemmas supply these inputs for the
quantitative gap and for BV minimization.

\begin{lemma}\label{lem:defectbalance}
For every quadruple in $\mathcal B$, the radial energy
\begin{equation}\label{eq:newtonenergybalance}
\mathcal E=
\int_0^\infty r^{n-1+\alpha}a_0v^{p+1}\,dr
=\int_0^\infty r^{n-1+\beta}b_0u^{q+1}\,dr
=\int_0^\infty r^{n-1}u'v'\,dr
\end{equation}
is finite and positive.
Put
\[
d_a=D(a_0),\quad d_b=D(b_0),\quad
\delta=\frac{d_a}{p+1}+\frac{d_b}{q+1}-\gamma_0>0,\quad
F_a=r^{d_a}a_0,\quad F_b=r^{d_b}b_0.
\]
Then $dF_a,dF_b$ are nonnegative and
\begin{equation}\label{eq:globaldefectbalance}
\delta\mathcal E
=\int_0^\infty\frac{r^{n+\alpha-d_a}v^{p+1}}{p+1}\,dF_a
+\int_0^\infty\frac{r^{n+\beta-d_b}u^{q+1}}{q+1}\,dF_b.
\end{equation}
For $P$ in \eqref{eq:systempohozaev}, with
$\ell=(n+\alpha-d_a)/(p+1)$ and $m=N-\ell$,
\begin{equation}\label{eq:pohozaevsmall}
|P(r)|\le\delta\mathcal E\qquad(r>0).
\end{equation}
\end{lemma}
\begin{proof}
Lemma~\ref{lem:systemdecay} bounds each energy density at infinity by
$Cr^{N-s-t-1}$, which is integrable.
At zero the bounds follow from $A,B>0$ and continuity of $u,v$.
Multiplication of the first radial equation by $v$ and integration by
parts gives the first and third terms in \eqref{eq:newtonenergybalance}.
The boundary term $r^{n-1}u'v$ vanishes at both endpoints.
Interchanging the equations gives the other equality.

Differentiate $P$ as in \eqref{eq:systempohozaevderivative}, now with
$\Gamma=-\delta$. Thus
\[
dP=
\frac{r^{n+\alpha-d_a}v^{p+1}}{p+1}\,dF_a
+\frac{r^{n+\beta-d_b}u^{q+1}}{q+1}\,dF_b
-\delta r^{n-1+\beta}b_0u^{q+1}\,dr.
\]
The bounds at zero and infinity give $P(0)=P(\infty)=0$.
Integration proves \eqref{eq:globaldefectbalance}.
Each of the positive and negative measures in this formula has total
mass $\delta\mathcal E$. Integration over $(0,r)$ proves
\eqref{eq:pohozaevsmall}.
\end{proof}

\begin{lemma}\label{lem:systemcompactness}
Suppose $\mathcal R(a_0,b_0)\le M$, where $M\ge1$.
Rescaling the two components and the radius gives
\begin{equation}\label{eq:systemnormalization}
u(0)=v(0)=1,\qquad 1\le a_0\le M,\qquad
\kappa\le b_0\le M\kappa,
\end{equation}
without changing the defects or coefficient ratios. Moreover,
\begin{equation}\label{eq:systemkappabounds}
\frac{c_-}{M}\le\kappa\le c_+ M^{B/A},
\end{equation}
where
\[
c_-=\frac{B(n+\beta)}
{2[2^p A(n+\alpha)]^{B/A}},\qquad
c_+=2^q B(n+\beta)
\left(\frac{2}{A(n+\alpha)}\right)^{B/A}.
\]
For fixed $M,D_0$, any normalized sequence with
$D(a_{0,j})+D(b_{0,j})\le D_0$ has a subsequence for which
\[
a_{0,j}\to a_0,\quad b_{0,j}\to b_0
\quad\hbox{in }L^1(0,R)\quad(R<\infty),
\]
and
\[
u_j\to u,\quad v_j\to v
\quad\hbox{locally uniformly on }[0,\infty)
\hbox{ and locally in }C^1((0,\infty)).
\]
The limiting coefficients are locally BV off zero and retain positive
upper and lower bounds. The limiting solutions are positive and satisfy
the flux equations. If $D(a_{0,j})\to d_a$ and $D(b_{0,j})\to d_b$,
then $r^{d_a}a_0$ and $r^{d_b}b_0$ have nondecreasing representatives.
\end{lemma}
\begin{proof}
Put $U_0=u(0)$, $V_0=v(0)$ and
\[
\widetilde u(r)=u(\lambda r)/U_0,\qquad
\widetilde v(r)=v(\lambda r)/V_0.
\]
The new coefficients are
\[
\widetilde a_0=\lambda^A V_0^p U_0^{-1}a_0(\lambda r),\qquad
\widetilde b_0=\lambda^B U_0^q V_0^{-1}b_0(\lambda r).
\]
Choose $\lambda$ so that $\inf\widetilde a_0=1$.
This proves \eqref{eq:systemnormalization}.
Since $u,v\le1$, the flux equations give
\[
u(r)\ge1-\frac{Mr^A}{A(n+\alpha)},\qquad
v(r)\ge1-\frac{M\kappa r^B}{B(n+\beta)}.
\]
Let $R_a=[A(n+\alpha)/(2M)]^{1/A}$. Then $u\ge1/2$ on $[0,R_a]$,
and positivity of $v(R_a)$ gives
\[
0<v(R_a)\le1-\frac{\kappa 2^{-q}R_a^B}{B(n+\beta)}.
\]
This is the upper bound in \eqref{eq:systemkappabounds}.
Let $R_b=[B(n+\beta)/(2M\kappa)]^{1/B}$.
Now $v\ge1/2$ on $[0,R_b]$, and
\[
0<u(R_b)\le1-\frac{2^{-p}R_b^A}{A(n+\alpha)}
\]
gives the lower bound.

For compactness, $r^{D_0}a_{0,j}$ and $r^{D_0}b_{0,j}$ are nondecreasing.
Helly's theorem on compact annuli and a diagonal extraction give
almost-everywhere convergence. The uniform coefficient bounds yield
$L^1(0,R)$ convergence by dominated convergence.
The flux equations give
$|u_j'|\le Cr^{A-1}$ and $|v_j'|\le Cr^{B-1}$ near zero.
They imply equicontinuity up to zero since $A,B>0$.
On compact annuli the second derivatives are uniformly bounded.
Arzel\`a--Ascoli gives the asserted solution convergence.
Dominated convergence passes the flux equations to the limit.
A zero of either limiting component would have strictly negative radial
derivative there, contrary to nonnegativity. Thus $u,v>0$.
Lemma~\ref{lem:systemdecay} applies uniformly.
Finally, pass the monotonicity of $r^{D(a_{0,j})}a_{0,j}$ and
$r^{D(b_{0,j})}b_{0,j}$ to the limit in distributions.
\end{proof}

\subsection{A quantitative lower bound for fixed coefficient ratios}

After normalization, a ratio bound keeps both solutions uniformly
positive on a short initial interval. The defect balance forces a
positive amount of coefficient variation there. Comparing it with
the total energy yields the following explicit gap.

\begin{proposition}\label{prop:systemquantlower}
Under \eqref{eq:systemradialrange}, there is $c>0$ such that
\begin{equation}\label{eq:systempolynomialgap}
\mathcal D(a_0,b_0)-\gamma_0\ge cM^{-\Theta}
\quad\hbox{whenever }\mathcal R(a_0,b_0)\le M.
\end{equation}
One admissible exponent is
\begin{equation}\label{eq:systemgapexponent}
\begin{aligned}
k_0&=\frac1A+\frac1B,\qquad
\xi_a=\frac{N(p+1)}{(pq-1)(s+t)},\qquad
\xi_b=\frac{N(q+1)}{(pq-1)(s+t)},\\
\theta_a&=\left[\frac{2(n+\alpha)}{(p+1)\gamma_0}-1\right]_+,\qquad
\theta_b=\left[\frac{2(n+\beta)}{(q+1)\gamma_0}-1\right]_+,\\
K_a&=k_0(n+\alpha)+\theta_a,\qquad
K_b=k_0(n+\beta)+\theta_b,\\
\Theta&=\min\{\xi_a+\max(K_a,1+K_b),\,
\xi_b+\max(K_b,1+K_a)\}.
\end{aligned}
\end{equation}
The constant $c$ depends only on the fixed parameters.
\end{proposition}
\begin{proof}
Use Lemma~\ref{lem:systemcompactness}.
The coefficient bounds and the proof of Lemma~\ref{lem:systemdecay} give
\[
u(r)\le C M^{p/(pq-1)}r^{-s},\qquad
v(r)\le C M^{1/(pq-1)}r^{-t}.
\]
Since $u,v\le1$,
\[
u(r)v(r)\le\min\{1,CM^{(p+1)/(pq-1)}r^{-(s+t)}\}.
\]
The flux bounds $-ru'\le Nu$ and $-rv'\le Nv$ imply
\begin{equation}\label{eq:quantenergyupper}
\begin{aligned}
\mathcal E
&\le N^2\int_0^\infty r^{N-1}u(r)v(r)\,dr\\
&\le C\int_0^\infty r^{N-1}
\min\{1,M^{(p+1)/(pq-1)}r^{-(s+t)}\}\,dr
\le CM^{\xi_a}.
\end{aligned}
\end{equation}
For the last inequality, split the integral at
$r=M^{(p+1)/((pq-1)(s+t))}$ and use $s+t>N$.
Choose a fixed $\eta\in(0,1)$ so small that
\[
\frac{\eta^A}{A(n+\alpha)}\le\frac12,\qquad
\frac{c_+\eta^B}{B(n+\beta)}\le\frac12.
\]
Set $R=\eta M^{-k_0}$. The lower bounds in the preceding proof imply
$u,v\ge1/2$ on $[0,R]$.

Write $d_a=D(a_0)$ and $d_b=D(b_0)$.
At least one of
\[
d_a\ge\frac{(p+1)\gamma_0}{2},\qquad
d_b\ge\frac{(q+1)\gamma_0}{2}
\]
holds. Suppose first that the former holds.
Set $r_0=R(2M)^{-1/d_a}$.
For $F_a=r^{d_a}a_0$,
\[
F_a(R)-F_a(r_0)\ge R^{d_a}/2.
\]
By \eqref{eq:globaldefectbalance},
\[
\begin{aligned}
\delta\mathcal E
&\ge
\int_{r_0}^R\frac{r^{n+\alpha-d_a}v^{p+1}}{p+1}\,dF_a\\
&\ge cR^{n+\alpha}
(2M)^{-(n+\alpha-d_a)_+/d_a}
\ge cM^{-k_0(n+\alpha)-\theta_a}.
\end{aligned}
\]
If instead the second inequality holds, use
$r_0=R(2M)^{-1/d_b}$ and
$F_b(R)-F_b(r_0)\ge\kappa R^{d_b}/2$.
Then
\[
\delta\mathcal E\ge
c\kappa R^{n+\beta}(2M)^{-(n+\beta-d_b)_+/d_b}
\ge cM^{-1-k_0(n+\beta)-\theta_b}.
\]
Combining these bounds with \eqref{eq:quantenergyupper} gives
\[
\mathcal D(a_0,b_0)-\gamma_0
\ge cM^{-\xi_a-\max(K_a,1+K_b)}.
\]
Interchanging the two equations gives the exponent
$\xi_b+\max(K_b,1+K_a)$.
Taking the smaller exponent proves \eqref{eq:systempolynomialgap}.
The constant may be replaced by the smaller of the two constants.
\end{proof}

\begin{remark}\label{rem:bvmechanism}
The qualitative compactness obstruction can also be read directly from
Lemma~\ref{lem:systemcompactness}. If the ratios stay bounded and
$\mathcal D(a_{0,j},b_{0,j})\to\gamma_0$, its limiting coefficients satisfy
$F_a=r^{d_a}a_0$, $F_b=r^{d_b}b_0$ nondecreasing and
$d_a/(p+1)+d_b/(q+1)=\gamma_0$.
The product rule for BV coefficients gives
\begin{equation*}
dP=
\frac{r^{n+\alpha-d_a}v^{p+1}}{p+1}\,dF_a
+\frac{r^{n+\beta-d_b}u^{q+1}}{q+1}\,dF_b\ge0.
\end{equation*}
The common origin and decay bounds give $P(0)=P(\infty)=0$.
Hence $dF_a=dF_b=0$. At least one of $d_a,d_b$ is positive,
so the resulting power coefficient cannot retain two-sided bounds.
This explains the compactness obstruction. Proposition~\ref{prop:systemquantlower}
quantifies it before passage to a limit.
\end{remark}

Combining this lower bound with the near-critical construction gives
constants $c,C>0$ and $M_0>1$, depending only on the fixed parameters,
such that
\begin{equation}\label{eq:systemcontrastbounds}
cM^{-\Theta}\le\mathcal D^{sys}_{\mathrm{crit}}(M)-\gamma_0
\le\frac{C}{\log M}\qquad(M\ge M_0).
\end{equation}
These rates need not be optimal. The lower estimate
\eqref{eq:systempolynomialgap} applies to each admissible quadruple
for every $M\ge1$.

\begin{proof}[Proof of the ratio bounds in Theorem~\ref{thm:systemcontrast}]
The lower bound follows from Proposition~\ref{prop:systemquantlower}.
Lemma~\ref{lem:systemconstruction} gives constants $C_1,C_2,T_0>0$ with
\[
\mathcal D(a_{0,T},b_{0,T})\le\gamma_0+C_1/T,\qquad
\mathcal R(a_{0,T},b_{0,T})\le e^{C_2T}.
\]
For $M\ge e^{C_2T_0}$ choose $T=(\log M)/C_2$.
This proves the upper bound. Rearranging \eqref{eq:systempolynomialgap} gives
$\mathcal R(a_0,b_0)\ge[c/(\mathcal D(a_0,b_0)-\gamma_0)]^{1/\Theta}$.
\end{proof}

\subsection{BV minimizers and smooth recovery}

The fixed-ratio normalization gives compactness of weighted monotone
coefficients, whose limits need only be locally BV. This is the natural
class in which to seek a minimizer. Relating it back to the smooth
problem requires a recovery procedure that preserves the equations,
even though it may slightly enlarge the coefficient ratio.

For a positive locally BV function with positive essential upper and lower
bounds, define
\[
D_{\mathrm{BV}}(c)=\inf\{\sigma\ge0:\frac{d}{dr}(r^\sigma c)\ge0\},\qquad
\mathcal R(c)=\frac{\mathop{\mathrm{ess\,sup}}c}
{\mathop{\mathrm{ess\,inf}}c}.
\]
The derivative inequality is understood in distributions on $(0,\infty)$.
The infimum defining a finite $D_{\mathrm{BV}}$ is attained, by passage to
the limit in this inequality. For a $C^1$ coefficient this definition agrees
with \eqref{eq:defectdefinition}.

Let $\mathcal B_{\mathrm{BV}}$ consist of such coefficients $a_0,b_0$ with
finite defects and positive radial solutions $u,v$, continuous at zero and
locally $C^1$ off zero, satisfying the flux equations
\eqref{eq:systemflux}. Their derivatives are locally absolutely continuous
off zero. Put
\[
\begin{aligned}
\mathcal D_{\mathrm{BV}}(a_0,b_0)
&=\frac{D_{\mathrm{BV}}(a_0)}{p+1}
+\frac{D_{\mathrm{BV}}(b_0)}{q+1},\\
\overline{\mathcal D}^{sys}_{\mathrm{crit}}(M)
&=\inf_{\substack{\mathcal B_{\mathrm{BV}}\\
\mathcal R(a_0,b_0)\le M}}
\mathcal D_{\mathrm{BV}}(a_0,b_0).
\end{aligned}
\]

\begin{theorem}\label{thm:BVminimizers}
Assume \eqref{eq:systemradialrange}. Whenever
$\overline{\mathcal D}_{\mathrm{crit}}(M)<\infty$, its infimum is attained in
$\mathcal B_{\mathrm{BV}}$ and
\[
\overline{\mathcal D}^{sys}_{\mathrm{crit}}(M)-\gamma_0\ge cM^{-\Theta}>0.
\]
For every $M\ge1$, with extended-real values allowed,
\begin{equation}\label{eq:relaxedrightenvelope}
\overline{\mathcal D}^{sys}_{\mathrm{crit}}(M)
=\lim_{h\downarrow0}\mathcal D^{sys}_{\mathrm{crit}}(M+h).
\end{equation}
Thus the relaxed value is nonincreasing and right-continuous.
It equals the original value at every point of right continuity of the
latter. On any interval where the original value is finite, equality fails
at at most countably many points.
\end{theorem}

The recovery lemma below smooths the solutions in logarithmic radius
and defines the coefficients from the equations. This preserves the
equations exactly and controls the defect and ratio errors.

\begin{lemma}\label{lem:dilationsmoothing}
Let $(a_0,b_0,u,v)\in\mathcal B_{\mathrm{BV}}$, with
$\mathcal R(a_0,b_0)\le M$ and
$D_{\mathrm{BV}}(a_0)+D_{\mathrm{BV}}(b_0)\le D_0$.
For every $0<\epsilon\le1$, there is a quadruple
$(a_\epsilon,b_\epsilon,u_\epsilon,v_\epsilon)\in\mathcal B$, smooth off zero,
such that
\begin{equation}\label{eq:BVrecoverybounds}
\begin{gathered}
\mathcal R(a_\epsilon,b_\epsilon)\le M e^{C\epsilon},\\
D(a_\epsilon)\le D_{\mathrm{BV}}(a_0)+C_{D_0}\epsilon,\qquad
D(b_\epsilon)\le D_{\mathrm{BV}}(b_0)+C_{D_0}\epsilon.
\end{gathered}
\end{equation}
The solutions converge locally uniformly up to zero and locally in $C^1$
off zero. The coefficients converge in local $L^1$ off zero.
Here $C$ depends only on the fixed equation parameters; $C_{D_0}$ also
depends on $D_0$.
\end{lemma}
\begin{proof}
The flux bounds give $0<x,y<N$, where
$x=-\frac{e^\tau u'(e^\tau)}{u(e^\tau)}$ and
$y=-\frac{e^\tau v'(e^\tau)}{v(e^\tau)}$.
These slopes are globally Lipschitz, with a constant depending only on
the parameters and $D_0$.
To verify this without coefficient differentiability, put
$z=e^{A\tau}a_0(e^\tau)v(e^\tau)^p/u(e^\tau)$.
Monotonicity of $e^{d_a\tau}a_0(e^\tau)$ gives
\[
z(\tau+t)\ge e^{-C t}z(\tau)\quad(0\le t\le1),\qquad
\int_\tau^{\tau+1}z\le N+N^2/4,
\]
for almost every $\tau$, with $d_a=D_{\mathrm{BV}}(a_0)$.
Thus $z\le C$. The identity $\dot{x}=z-x(N-x)$ bounds $\dot{x}$.
The same argument applies to $y$.

Take a nonnegative smooth mollifier $\rho_\epsilon$, supported in
$(-\epsilon,\epsilon)$, with integral one. Define
\[
u_\epsilon(r)=\int\rho_\epsilon(t)u(e^tr)\,dt,
\qquad
v_\epsilon(r)=\int\rho_\epsilon(t)v(e^tr)\,dt,
\]
and
\[
\begin{aligned}
a_\epsilon(r)&=
\frac{\int\rho_\epsilon(t)e^{At}a_0(e^tr)v(e^tr)^p\,dt}
{v_\epsilon(r)^p},\\
b_\epsilon(r)&=
\frac{\int\rho_\epsilon(t)e^{Bt}b_0(e^tr)u(e^tr)^q\,dt}
{u_\epsilon(r)^q}.
\end{aligned}
\]
Linearity and scaling of the Laplacian give both equations exactly.
Logarithmic convolution makes all four profiles smooth off zero.
The solution values at zero are unchanged.
The slope bounds imply
\[
e^{-N\epsilon}\le\frac{v(e^tr)}{v(r)}\le e^{N\epsilon}
\quad(|t|\le\epsilon).
\]
Consequently, with $C_a=A+2pN$,
\[
e^{-C_a\epsilon}\mathop{\mathrm{ess\,inf}}a_0
\le a_\epsilon(r)\le
e^{C_a\epsilon}\mathop{\mathrm{ess\,sup}}a_0.
\]
The analogous bound uses $C_b=B+2qN$.
This proves the ratio estimate in \eqref{eq:BVrecoverybounds}.

In logarithmic coordinates, let $f_\epsilon$ be the numerator defining
$a_\epsilon$, and abbreviate $v_t=v(e^{\tau+t})$ and
$y_t=y(\tau+t)$.
The inequality $\frac{d}{d\tau}a_0(e^\tau)+d_a a_0(e^\tau)\ge0$, with $a_0$ read in
logarithmic coordinates, yields
\[
\frac{\dot{f}_\epsilon}{f_\epsilon}
\ge-d_a-p
\frac{\int\rho_\epsilon(t)e^{At}a_0(e^{\tau+t})v_t^p y_t\,dt}
{\int\rho_\epsilon(t)e^{At}a_0(e^{\tau+t})v_t^p\,dt}.
\]
Also
\[
\frac{e^\tau v_\epsilon'(e^\tau)}{v_\epsilon(e^\tau)}
=-\frac{\int\rho_\epsilon(t)v_t y_t\,dt}
{\int\rho_\epsilon(t)v_t\,dt}.
\]
Both averages of $y_t$ lie between its minimum and maximum on
$[\tau-\epsilon,\tau+\epsilon]$.
Its Lipschitz bound therefore gives
\[
\frac{e^\tau a_\epsilon'(e^\tau)}{a_\epsilon(e^\tau)}
\ge-d_a-2p\|\dot{y}\|_\infty\epsilon.
\]
The second defect estimate is identical.
Differentiating the convolution kernel instead gives a finite global bound
for the absolute logarithmic derivatives, for each fixed $\epsilon$.
The positive coefficient bounds and the flux equations prove membership
in $\mathcal B$. Approximation by logarithmic convolution gives the stated
convergences.
\end{proof}

\begin{proof}[Proof of Theorem~\ref{thm:BVminimizers}]
Choose a minimizing sequence with bounded combined defects.
The normalization and the $\kappa$ bounds in
Lemma~\ref{lem:systemcompactness} use only the flux equations.
They therefore hold in $\mathcal B_{\mathrm{BV}}$.
The proof of that lemma gives a subsequence with local $L^1$ coefficient
convergence and local uniform solution convergence up to zero.
The limiting coefficients retain positive upper and lower bounds and ratios
at most $M$. The limiting solutions satisfy the flux equations and are
positive.
Extract the individual defect limits $d_a,d_b$.
Passing the monotonicity of $r^{d_{a,j}}a_{0,j}$ and
$r^{d_{b,j}}b_{0,j}$ to the limit gives
\[
D_{\mathrm{BV}}(a_0)\le d_a,\qquad D_{\mathrm{BV}}(b_0)\le d_b.
\]
This proves lower semicontinuity and attainment.

The BV product rule gives the global defect balance
\eqref{eq:globaldefectbalance} with nonnegative Stieltjes measures.
The origin and infinity boundary terms vanish by the same flux and decay
bounds as before. To exclude a combined defect at most $\gamma_0$, choose
$\widetilde d_a\ge D_{\mathrm{BV}}(a_0)$ and
$\widetilde d_b\ge D_{\mathrm{BV}}(b_0)$ with
$\widetilde d_a/(p+1)+\widetilde d_b/(q+1)=\gamma_0$.
The associated Pohozaev function is nondecreasing and vanishes at both
endpoints. Its positive variation measures therefore vanish.
Thus $r^{\widetilde d_a}a_0$ and $r^{\widetilde d_b}b_0$ are constant.
At least one exponent is positive, which contradicts the coefficient bounds.
The proof of Proposition~\ref{prop:systemquantlower}, already written in
Stieltjes form, gives the asserted quantitative lower bound.

Write $g(M)=\mathcal D^{sys}_{\mathrm{crit}}(M)$.
If the relaxed value is finite, apply Lemma~\ref{lem:dilationsmoothing}
to a minimizer. For each fixed $h>0$, sufficiently small $\epsilon$ gives
$\mathcal R(a_\epsilon,b_\epsilon)\le M+h$ and
\[
g(M+h)\le\overline{\mathcal D}^{sys}_{\mathrm{crit}}(M)+C\epsilon.
\]
Let $\epsilon\downarrow0$ and then $h\downarrow0$.
This proves the upper inequality in \eqref{eq:relaxedrightenvelope}.

Conversely, suppose $\lim_{h\downarrow0}g(M+h)<\infty$.
Take $h_j\downarrow0$ and quadruples with ratio at most $M+h_j$ and
combined defect at most $g(M+h_j)+j^{-1}$.
Their normalized coefficients have common positive bounds, their defects
are bounded, and the preceding compactness argument gives a limit in
$\mathcal B_{\mathrm{BV}}$ with ratio at most $M$.
Lower semicontinuity gives
\[
\overline{\mathcal D}^{sys}_{\mathrm{crit}}(M)
\le\lim_{j\to\infty}g(M+h_j).
\]
It also shows that a finite right limit is impossible if the relaxed class
is empty. This proves \eqref{eq:relaxedrightenvelope} in all cases.
The right envelope of a nonincreasing function is right-continuous.
A finite monotone function has at most countably many jumps, proving the
remaining assertions.
\end{proof}

The inclusion of the smooth class in the relaxed class and
\eqref{eq:relaxedrightenvelope} give
\[
\overline{\mathcal D}^{sys}_{\mathrm{crit}}(M)
=\mathcal D^{sys}_{\mathrm{crit}}(M+)
\le\mathcal D^{sys}_{\mathrm{crit}}(M).
\]
Recovery gives ratios at most $M+h$ for each fixed $h>0$.
It need not preserve the exact constraint $M$.
Thus neither attainment in the original $C^1$ class nor equality at a
right jump follows from Lemma~\ref{lem:dilationsmoothing}.
The scalar class $\mathcal A_{n,p}$ also requires regularity at zero.
Theorem~\ref{thm:BVminimizers} does not assert a minimizer in that class.

\section{System limits and logarithmic spreading}\label{sec:limits}

Throughout this section, assume \eqref{eq:systemradialrange}.
The fixed-ratio gap forces coefficient degeneration near the threshold.
We now normalize at a peak of the component product to obtain local
limits without a common coefficient-ratio bound. The product of the
defect excess and the normalized energy controls local homogeneity;
the energy bound itself determines whether the limit has finite energy.

For a quadruple in $\mathcal B$, define
\begin{equation}\label{eq:concentrationheight}
\mathcal E=\int_0^\infty r^{n-1}u'v'\,dr,\qquad
Q(r)=r^Nu(r)v(r),\qquad Z=\max_{r>0}Q(r),\qquad H=\mathcal E/Z.
\end{equation}
Lemmas~\ref{lem:systemdecay} and~\ref{lem:defectbalance} give
$0<\mathcal E,Z<\infty$ and $Q(0)=Q(\infty)=0$, so the maximum is attained.
The quantity $H$ is invariant under spatial dilation and constant
rescaling of either component. For a sequence choose any $r_j>0$ with
$r_j^Nu_j(r_j)v_j(r_j)=Z_j$, and set
\begin{equation}\label{eq:peaknormalization}
\begin{aligned}
U_j(r)&=\frac{u_j(r_jr)}{u_j(r_j)},&
V_j(r)&=\frac{v_j(r_jr)}{v_j(r_j)},\\
A_j(r)&=r_j^A\frac{v_j(r_j)^p}{u_j(r_j)}a_{0,j}(r_jr),&
B_j(r)&=r_j^B\frac{u_j(r_j)^q}{v_j(r_j)}b_{0,j}(r_jr).
\end{aligned}
\end{equation}
The normalized equations have coefficient factors $A_j,B_j$, their
energy is $H_j$, and $U_j(1)=V_j(1)=1$, $r^NU_jV_j\le1$.
The full local limit theorem, including the general BV limit before
an energy restriction is imposed, is as follows.

\begin{theorem}\label{thm:sequencelimits}
Let $(a_{0,j},b_{0,j},u_j,v_j)\in\mathcal B$ satisfy
\[
\delta_j=\mathcal D(a_{0,j},b_{0,j})-\gamma_0\longrightarrow0.
\]
Use the peak normalization \eqref{eq:peaknormalization}.
After passage to a subsequence,
\[
D(a_{0,j})\to d_a,\qquad D(b_{0,j})\to d_b,\qquad
\frac{d_a}{p+1}+\frac{d_b}{q+1}=\gamma_0.
\]
Set $\ell=(n+\alpha-d_a)/(p+1)$ and $m=(n+\beta-d_b)/(q+1)$.
Then $\ell,m>0$ and $\ell+m=N$.
There are nonnegative locally BV coefficients $a_\infty,b_\infty$ and
positive radial $U,V$ such that
\[
\begin{gathered}
A_j\to a_\infty,\quad B_j\to b_\infty
\quad\hbox{in }L^1_{\mathrm{loc}}((0,\infty)),\\
U_j\to U,\quad V_j\to V
\quad\hbox{in }C^1_{\mathrm{loc}}((0,\infty)),\\
-\Delta U=r^\alpha a_\infty V^p,\qquad
-\Delta V=r^\beta b_\infty U^q
\quad\hbox{in distributions},\\
U(1)=V(1)=1,\qquad r^NU(r)V(r)\le1.
\end{gathered}
\]
The functions $F_a=r^{d_a}a_\infty$ and $F_b=r^{d_b}b_\infty$ have
nondecreasing representatives. Define
\begin{equation}\label{eq:limitpohozaevfunction}
\begin{aligned}
P={}&r^n\left(U'V'+\frac{r^\alpha a_\infty V^{p+1}}{p+1}
+\frac{r^\beta b_\infty U^{q+1}}{q+1}\right)\\
&+r^{n-1}(mUV'+\ell VU').
\end{aligned}
\end{equation}
It has a bounded nondecreasing representative and satisfies
\begin{equation}\label{eq:limitpohozaevmeasure}
dP=\frac{r^{n+\alpha-d_a}V^{p+1}}{p+1}\,dF_a
+\frac{r^{n+\beta-d_b}U^{q+1}}{q+1}\,dF_b\ge0.
\end{equation}
Precisely one of the following alternatives holds along a further subsequence.

\smallskip\noindent{\rm(i)} The sequence $H_j$ is bounded.
There are $c_a,c_b>0$ such that
\[
A_j\to c_ar^{-d_a},\qquad B_j\to c_br^{-d_b}
\quad\hbox{locally uniformly on }(0,\infty),
\]
and $U_j\to U$, $V_j\to V$ in $C^2_{\mathrm{loc}}((0,\infty))$.
The limit solves
\[
-\Delta U=c_ar^{\alpha-d_a}V^p,\qquad
-\Delta V=c_br^{\beta-d_b}U^q.
\]
Its two nonlinear energies are finite and $P\equiv0$.
No regularity at the origin is asserted.

\smallskip\noindent{\rm(ii)} One has $H_j\to\infty$.
For every finite $L>0$,
\begin{equation}\label{eq:uniformlogvanishing}
\sup_{h\in\mathbb R}\frac1{2\mathcal E_j}
\int_{e^{h-L}}^{e^{h+L}}
\left(r^{n-1+\alpha}a_{0,j}v_j^{p+1}
+r^{n-1+\beta}b_{0,j}u_j^{q+1}\right)\,dr
\longrightarrow0.
\end{equation}

If $\delta_jH_j\to0$, the homogeneous coefficient limits, local $C^2$
convergence and identity $P=0$ in {\rm(i)} also hold without a bound on $H_j$.
The finite-energy conclusion is asserted only under that bound.
In both alternatives,
\[
\mathcal R(a_{0,j},b_{0,j})
\ge(c/\delta_j)^{1/\Theta}\longrightarrow\infty,
\]
and Proposition~\ref{prop:measurerigidity} gives energy-weighted defect rigidity.
\end{theorem}

We begin the proof with three consequences of the defect balance.
First, normalizing each nonlinear energy to total mass one measures
where the coefficient slope differs from its extremal allowed value.
This gives rigidity in energy-weighted measure without assuming
a bound on the coefficient ratios.

\begin{proposition}\label{prop:measurerigidity}
Put $d_a=D(a_0)$, $d_b=D(b_0)$ and
$\delta=d_a/(p+1)+d_b/(q+1)-\gamma_0$.
Define probability measures on $(0,\infty)$ by
\[
d\nu_a=\mathcal E^{-1}r^{n-1+\alpha}a_0v^{p+1}\,dr,\qquad
d\nu_b=\mathcal E^{-1}r^{n-1+\beta}b_0u^{q+1}\,dr.
\]
Then
\begin{equation*}
\frac1{p+1}\int\left(d_a+\frac{ra_0'}{a_0}\right)d\nu_a
+\frac1{q+1}\int\left(d_b+\frac{rb_0'}{b_0}\right)d\nu_b
=\delta.
\end{equation*}
In particular, for every $\varepsilon>0$,
\[
\nu_a\left\{d_a+\frac{ra_0'}{a_0}\ge\varepsilon\right\}
\le\frac{(p+1)\delta}{\varepsilon},\qquad
\nu_b\left\{d_b+\frac{rb_0'}{b_0}\ge\varepsilon\right\}
\le\frac{(q+1)\delta}{\varepsilon}.
\]
These conclusions apply to every near-threshold sequence, without a
common coefficient ratio or a common upper logarithmic derivative bound.
\end{proposition}
\begin{proof}
Divide \eqref{eq:globaldefectbalance} by $\mathcal E$ and use
$d(r^{d_a}a_0)=r^{d_a-1}a_0(d_a+ra_0'/a_0)\,dr$.
Both integrands are nonnegative. Chebyshev's inequality gives the bounds.
\end{proof}

For local compactness and spreading we also need pointwise control
of the logarithmic slopes. The next bounds do not require a ratio
constraint; after peak normalization they bound the energy carried
by any interval of fixed logarithmic length.

\begin{lemma}\label{lem:phasecompactness}
Suppose $D(a_0)+D(b_0)\le D_0$.
With $\tau=\log r$, put
\[
x=-\frac{e^\tau u'(e^\tau)}{u(e^\tau)},\qquad y=-\frac{e^\tau v'(e^\tau)}{v(e^\tau)},
\]
\[
z=r^Aa_0v^p/u,\qquad w=r^Bb_0u^q/v.
\]
There is $C$, depending only on the parameters and $D_0$, such that
\begin{equation}\label{eq:uniformphase}
0<x,y<N,\qquad 0<z,w\le C,\qquad |\dot{x}|+|\dot{y}|\le C.
\end{equation}
Moreover $H\ge N$. For every $L>0$,
\begin{equation}\label{eq:logenergybound}
\sup_{h\in\mathbb R}
\frac1{2\mathcal E}
\int_{e^{h-L}}^{e^{h+L}}
\left(r^{n-1+\alpha}a_0v^{p+1}
+r^{n-1+\beta}b_0u^{q+1}\right)\,dr
\le\frac{C_L}{H}.
\end{equation}
If the quadruples are normalized by
\[
u(1)=v(1)=1,\qquad Q(r)\le Q(1)=1,
\]
then their coefficients are uniformly bounded above on compact annuli.
Their solutions are precompact in $C^1_{\mathrm{loc}}((0,\infty))$.
After passage to a subsequence, the coefficients converge in
$L^1_{\mathrm{loc}}((0,\infty))$ to nonnegative locally BV coefficients.
\end{lemma}
\begin{proof}
The flux formulas and the zero limits of $u,v$ at infinity give
$0<-ru'/u<N$ and $0<-rv'/v<N$.
The strict upper bounds follow because both fluxes are strictly increasing.
Also
\[
\dot{x}=z-x(N-x),\qquad \dot{y}=w-y(N-y),
\]
and
\[
\frac{\dot{z}}z=A+x-py+\frac{e^\tau a_0'(e^\tau)}{a_0(e^\tau)}
\ge A-pN-D_0.
\]
Let $C_0=(pN+D_0-A)_+$. On $[\tau,\tau+1]$,
$z(t)\ge z(\tau)e^{-C_0}$, whereas
\[
\int_\tau^{\tau+1}z(t)\,dt
=x(\tau+1)-x(\tau)
+\int_\tau^{\tau+1}x(N-x)\,dt
\le N+N^2/4.
\]
This bounds $z$. The same argument bounds $w$.
The derivative estimates follow.

Write $F(r)=-r^{n-1}u'(r)$ and $G(r)=-r^{n-1}v'(r)$.
They are positive and nondecreasing, and
$u(r)=\int_r^\infty F(t)t^{-N-1}\,dt$.
Chebyshev's integral inequality for the probability measure
$Nr^Nt^{-N-1}\,dt$ on $[r,\infty)$ gives
\[
\mathcal E\ge
\int_r^\infty F(t)G(t)t^{-N-1}\,dt
\ge Nr^Nu(r)v(r).
\]
Taking the maximum proves $H\ge N$.
In logarithmic radius the two energy densities are $Qz$ and $Qw$.
Since $Q\le Z$, \eqref{eq:uniformphase} gives \eqref{eq:logenergybound}.

Under the stated normalization, integration of $0\le x,y\le N$ gives
\[
\begin{cases}
1\le u(r),v(r)\le r^{-N},&0<r\le1,\\
r^{-N}\le u(r),v(r)\le1,&r\ge1.
\end{cases}
\]
Thus $a_0=r^{-A}uz/v^p$ and $b_0=r^{-B}vw/u^q$ are uniformly bounded
above on compact annuli. The radial equations bound the second derivatives
there. Arzel\`a--Ascoli gives local $C^1$ compactness.
The functions $r^{D_0}a_0$ and $r^{D_0}b_0$ are nondecreasing.
Helly's theorem gives coefficient convergence almost everywhere and in
local $L^1$. The limits are nonnegative and locally BV.
\end{proof}

The phase bounds ensure that both normalized components stay
uniformly positive on compact annuli. On such an annulus the weighted
variation identity therefore controls the ordinary variation of the
coefficients after their limiting powers have been removed.

\begin{proposition}\label{prop:localrigidity}
Let $(a_0,b_0,u,v)\in\mathcal B$ satisfy
\[
u(1)=v(1)=1,\qquad r^Nu(r)v(r)\le1,\qquad
d_a+d_b\le D_0,
\]
where $d_a=D(a_0)$ and $d_b=D(b_0)$.
Set $\delta=\mathcal D(a_0,b_0)-\gamma_0$ and $I_L=[e^{-L},e^L]$.
For every $L>0$ there is $C_L$, depending only on the parameters,
$D_0$ and $L$, such that
\begin{equation}\label{eq:localcoefficientrigidity}
\|r^{d_a}a_0-a_0(1)\|_{L^\infty(I_L)}
+\|r^{d_b}b_0-b_0(1)\|_{L^\infty(I_L)}
\le C_L\delta H.
\end{equation}
\end{proposition}
\begin{proof}
Here $Z=1$ and $\mathcal E=H$.
The slope bounds in Lemma~\ref{lem:phasecompactness} give
$u,v\ge e^{-NL}$ on $I_L$.
Since $0\le d_a,d_b\le D_0$, both weights
\[
\frac{r^{n+\alpha-d_a}v^{p+1}}{p+1},\qquad
\frac{r^{n+\beta-d_b}u^{q+1}}{q+1}
\]
are bounded below by a positive constant depending only on the stated data.
With $F_a=r^{d_a}a_0$ and $F_b=r^{d_b}b_0$,
\eqref{eq:globaldefectbalance} yields
\[
F_a(e^L)-F_a(e^{-L})+F_b(e^L)-F_b(e^{-L})
\le C_L\delta H.
\]
Both functions are nondecreasing. This proves
\eqref{eq:localcoefficientrigidity}.
\end{proof}

\begin{remark}
The controlling quantity in \eqref{eq:localcoefficientrigidity} is
$\delta H$. Thus $\delta_jH_j\to0$ forces local coefficient rigidity even
when $H_j$ is unbounded. A bound on $H_j$ is used separately to obtain
finite energy of the limiting profile.
\end{remark}

\begin{proof}[Proof of Theorem~\ref{thm:sequencelimits}]
The individual defects are bounded, since their weighted sum tends to
$\gamma_0$. Extract their limits $d_a,d_b$.
The numbers
\[
\ell_0=\frac{N+A}{p+1},\qquad m_0=\frac{N+B}{q+1}
\]
are less than $N$, because $A=pt-s<pN$ and $B=qs-t<qN$.
Thus the limiting values
\[
\ell=\frac{n+\alpha-d_a}{p+1},\qquad
m=\frac{n+\beta-d_b}{q+1}
\]
satisfy $\ell+m=N$ and
\[
\ell\ge N-m_0>0,\qquad m\ge N-\ell_0>0.
\]
All subsequent quadruples are normalized by \eqref{eq:peaknormalization}.
Their energies equal $H_j$ and $r^NU_jV_j\le1$.

\smallskip\noindent\emph{Local limits without an energy bound.}
Lemma~\ref{lem:phasecompactness} gives, after extraction,
\[
U_j,V_j\longrightarrow U,V\quad\hbox{in }C^1_{\mathrm{loc}}((0,\infty)),
\qquad A_j,B_j\longrightarrow a_\infty,b_\infty
\quad\hbox{in }L^1_{\mathrm{loc}}((0,\infty)).
\]
The limiting solutions are positive and locally $W^{2,\infty}$.
The radial equations pass to the limit in distributions.
Write $d_{a,j}=D(a_{0,j})$, $d_{b,j}=D(b_{0,j})$ and
\[
F_{a,j}=r^{d_{a,j}}A_j,\qquad F_{b,j}=r^{d_{b,j}}B_j.
\]
These functions are nondecreasing and locally uniformly bounded.
Their local $L^1$ limits are $F_a=r^{d_a}a_\infty$ and
$F_b=r^{d_b}b_\infty$. Their derivative measures converge weakly on
compact intervals. In particular, $dF_a,dF_b\ge0$.

Let $P_j$ be the Pohozaev function in \eqref{eq:pohozaevsmall}.
It converges in local $L^1$ to the expression
\eqref{eq:limitpohozaevfunction}.
The density of the negative term in its derivative is
\[
\delta_j r^{n-1+\beta}B_jU_j^{q+1}.
\]
Its integral on every compact annulus tends to zero, by
Lemma~\ref{lem:phasecompactness}.
The other two weights converge locally uniformly.
The measure identity in the proof of Lemma~\ref{lem:defectbalance}
therefore passes to the limit and gives \eqref{eq:limitpohozaevmeasure}.
For completeness, the phase expression is
\[
P_j=Q_j\left(x_jy_j+\frac{z_j}{p+1}+\frac{w_j}{q+1}
-m_jy_j-\ell_jx_j\right).
\]
Here $0<Q_j\le1$, the phases are uniformly bounded, and
$\ell_j,m_j$ are bounded. Hence $|P_j|\le C$ on $(0,\infty)$.
The limit $P$ has a bounded nondecreasing representative.
Its finite endpoint limits exist. Integrating
\eqref{eq:limitpohozaevmeasure} gives
\begin{equation}\label{eq:limitboundaryvariation}
\begin{aligned}
P(\infty)-P(0+)
&=\int_0^\infty\frac{r^{n+\alpha-d_a}V^{p+1}}{p+1}\,dF_a\\
&\quad+\int_0^\infty\frac{r^{n+\beta-d_b}U^{q+1}}{q+1}\,dF_b<\infty.
\end{aligned}
\end{equation}
The identity \eqref{eq:limitboundaryvariation} does not identify either
endpoint value with zero.

\smallskip\noindent\emph{Vanishing normalized Pohozaev error.}
Suppose $\eta_j=\delta_jH_j\to0$.
Proposition~\ref{prop:localrigidity} gives
\[
\|r^{d_{a,j}}A_j-A_j(1)\|_{L^\infty(I_L)}
+\|r^{d_{b,j}}B_j-B_j(1)\|_{L^\infty(I_L)}
\le C_L\eta_j\longrightarrow0.
\]
The coefficient values at one are bounded.
After extraction, they tend to $c_a,c_b\ge0$.
Consequently
\[
A_j\to c_ar^{-d_a},\qquad B_j\to c_br^{-d_b}
\quad\hbox{locally uniformly}.
\]
The radial equations give local $C^2$ convergence of the solutions.
Also $|P_j|\le\eta_j$ by \eqref{eq:pohozaevsmall}, so $P\equiv0$.

We exclude $c_ac_b=0$ without assuming bounded energy.
Suppose $c_a=0$. Positivity, monotonicity and $U(1)=1$ imply
\[
U=c_0+c_1r^{-N},\qquad c_0,c_1\ge0,\qquad c_0+c_1=1.
\]
Assume $c_b>0$ and write $G=-r^{n-1}V'\ge0$.
If $c_0>0$, the second equation gives, for large $r$,
\[
G(r)\ge G(r)-G(r/2)\ge c r^{n+\beta-d_b}.
\]
Integrating $-V'$ on $[r,2r]$ yields
$V(r)\ge c r^{B-d_b}$.
It follows that
\[
r^NU(r)V(r)\ge c r^{n+\beta-d_b}
=c r^{(q+1)m}\longrightarrow\infty,
\]
contrary to the normalization.
If $c_0=0$, then $U=r^{-N}$ and $V\le1$.
Put $k=(q+1)\ell>0$. The second equation becomes
$-\Delta V=c_br^{-2-k}$.
For small $r$,
\[
G(2r)\ge G(2r)-G(r)\ge c r^{N-k},\qquad
V(2r)-V(4r)\ge c r^{-k},
\]
again a contradiction. Hence $c_b=0$.
Both components are then radial harmonic functions.
The conditions $U(1)=V(1)=1$ and $r^NUV\le1$ force
$(U,V)=(1,r^{-N})$ or $(r^{-N},1)$.
Their Pohozaev values are $-mN$ and $-\ell N$, respectively.
Both contradict $P=0$.
Interchanging the equations excludes $c_b=0$, and therefore $c_a,c_b>0$.

\smallskip\noindent\emph{The energy alternatives.}
Pass to a subsequence on which $H_j$ is bounded or tends to infinity.
If $H_j\le H_0$, then $\eta_j\to0$ and the preceding argument applies.
Fatou's lemma gives
\[
\int_0^\infty
\left(c_ar^{n-1+\alpha-d_a}V^{p+1}
+c_br^{n-1+\beta-d_b}U^{q+1}\right)\,dr\le2H_0.
\]
This proves alternative {\rm(i)}.
If $H_j\to\infty$, equation \eqref{eq:logenergybound} gives
\eqref{eq:uniformlogvanishing}, proving alternative {\rm(ii)}.
The coefficient-ratio bound follows from
Proposition~\ref{prop:systemquantlower}.
Proposition~\ref{prop:measurerigidity} applies to both alternatives.
\end{proof}

The bounded-energy and spreading conclusions prove the overview in
Theorem~\ref{thm:systemoverview}. The additional condition
$\delta_jH_j\to0$ concerns local rigidity and can hold in either
energy regime.

\begin{remark}\label{rem:limitsscope}
In Theorem~\ref{thm:sequencelimits}, the effective weights
$\alpha-d_a,\beta-d_b$ may be at most $-2$; regularity at zero is not asserted.
Equation \eqref{eq:limitboundaryvariation} gives the endpoint Pohozaev
difference. If $\delta_jH_j\to0$, then $P=0$ and the coefficients are
positive homogeneous functions. Bounded $H_j$ gives finite limiting energy.
For scalar near-threshold sequences with bounded normalized energy,
Theorems~\ref{thm:scalarprofiles} and~\ref{thm:energyrealization} describe
all nonzero profiles and all possible remaining energies.
Corollary~\ref{cor:coefficientmoments} gives a sufficient moment condition
for finite strong decomposition. Remark~\ref{rem:compactnesslevels}
separates this conclusion from single-profile convergence. A complete profile decomposition for unequal systems
is not proved here. Their cross energy $\int r^{n-1}u'v'\,dr$ is not a
positive definite squared norm on arbitrary variations.

The lower and upper coefficient-ratio exponents do not match.
Their optimal asymptotics remain undetermined.
Proposition~\ref{prop:linearsharpness} proves the optimal linear order of
the scalar stability error and of $H_*-\Phi(\varepsilon)$.
It does not identify the best stability constant or a first-order
asymptotic coefficient for $\Phi$.
Theorem~\ref{thm:BVminimizers} identifies the relaxed value with the smooth
right envelope. It leaves open attainment in the original $C^1$ class and
equality of the two values at every fixed ratio.
\end{remark}

\appendix
\section{Radial Newton representation and admissible exponents}\label{sec:newtonbackground}

In this appendix $n\ge3$, $p,q>0$, $\alpha,\beta>-2$,
$N=n-2$, $A=2+\alpha$ and $B=2+\beta$.
The radial coefficient factors are $C^1$ off zero and satisfy
positive upper and lower bounds.
No subcriticality condition is imposed.

\begin{lemma}\label{lem:newton}
Let $u,v$ be positive radial functions in
$C(\mathbb R^n)\cap C^2(\mathbb R^n\setminus\{0\})$ satisfying
\[
-\Delta u=r^\alpha a_0v^p,\qquad
-\Delta v=r^\beta b_0u^q\quad(r>0).
\]
Then $u(\infty)=v(\infty)=0$, and
\[
u=\Phi*(|\cdot|^\alpha a_0v^p),\qquad
v=\Phi*(|\cdot|^\beta b_0u^q),\qquad
\Phi(x)=\frac{|x|^{-N}}{N|\mathbb S^{n-1}|}.
\]
Both potentials are finite at every point, including zero.
\end{lemma}
\begin{proof}
Continuity at zero eliminates the constants of integration in the radial
flux equations, as in Lemma~\ref{lem:systemdecay}.
Both components decrease and have nonnegative limits at infinity.
If $u(\infty)>0$, then the second flux equation gives
$-v'(r)\ge cr^{B-1}$ for sufficiently large $r$.
Since $B>0$, this contradicts positivity of $v$.
The same argument gives $v(\infty)=0$.

Put $g(r)=r^\alpha a_0(r)v(r)^p$.
Integration of the first flux equation and Tonelli's theorem give
\[
u(r)=\frac1N\left(
r^{-N}\int_0^r \xi^{n-1}g(\xi)\,d\xi
+\int_r^\infty \xi g(\xi)\,d\xi\right)\qquad(r>0).
\]
The right side is the radial Newton potential, by integration over spheres.
Letting $r\downarrow0$ in its equivalent flux integral gives
$u(0)=N^{-1}\int_0^\infty \xi g(\xi)\,d\xi<\infty$.
The proof for $v$ is identical.
\end{proof}

\begin{proposition}\label{prop:existencerange}
A positive radial entire solution with coefficient factors as above can
exist only if
\[
pq>1,\qquad
s=\frac{A+pB}{pq-1}<N,\qquad
t=\frac{B+qA}{pq-1}<N.
\]
Conversely, these conditions imply existence for some choice of positive
radial coefficient factors with bounded logarithmic derivatives.
\end{proposition}
\begin{proof}
The flux inequalities used in Lemma~\ref{lem:systemdecay} give
\[
u(r)\ge cr^Av(r)^p,\qquad v(r)\ge cr^Bu(r)^q,
\qquad u(r)^{1-pq}\ge cr^{A+pB}.
\]
If $pq\le1$, the last inequality contradicts $u(r)\le u(0)$ as
$r\to\infty$.
For $pq>1$, substitution gives $u(r)\le Cr^{-s}$ and $v(r)\le Cr^{-t}$.
Each flux is bounded below by a positive constant for $r\ge1$.
The representation in Lemma~\ref{lem:newton} therefore gives
$u(r),v(r)\ge cr^{-N}$ for $r\ge1$.
Hence $s,t\le N$.

Suppose $s=N$.
The bound $-ru'(r)\le Nu(r)$ implies
$-r^{n-1}u'(r)\le C$ for $r\ge1$.
On the other hand,
\[
v(r)\ge cr^Bu(r)^q\ge cr^{B-qN}=cr^{-t}.
\]
Since $pt=s+A=N+A$, one has
$r^\alpha a_0v^p\ge cr^{-n}$ for $r\ge1$.
The first flux then grows at least as $c\log r$, a contradiction.
Interchanging the equations excludes $t=N$.

For sufficiency, assume $0<s,t<N$ and set
\[
u(r)=(1+r^A)^{-s/A},\qquad v(r)=(1+r^B)^{-t/B},
\]
\[
\begin{aligned}
a_0(r)&=s[N+A+(N-s)r^A]
(1+r^A)^{-s/A-2}(1+r^B)^{pt/B},\\
b_0(r)&=t[N+B+(N-t)r^B]
(1+r^B)^{-t/B-2}(1+r^A)^{qs/A}.
\end{aligned}
\]
Direct differentiation verifies both equations.
The identities $pt=s+A$ and $qs=t+B$ give
\[
\begin{aligned}
a_0(0)&=s(N+A),&a_0(\infty)&=s(N-s),\\
b_0(0)&=t(N+B),&b_0(\infty)&=t(N-t).
\end{aligned}
\]
All four limits are positive.
The logarithmic derivatives are bounded rational expressions in $r^A,r^B$.
Thus the coefficients have the required bounds.
Lemma~\ref{lem:newton} gives the finite Newton potentials.
\end{proof}

\begin{proposition}\label{prop:counterexample}
Assume \eqref{eq:radialrange} and put $\mu=2/(p-1)$.
The pair from~\cite{LeiLi},
\[
U(r)=(1+r^2)^{-\mu/2},\qquad
c(r)=\mu\frac{n+(N-\mu)r^2}{1+r^2},
\]
belongs to $\mathcal A_{n,p}$. Its ratio and defect are
\[
\mathcal R(c)=\frac n{N-\mu},\qquad
D(c)=\frac{2(\sqrt n-\sqrt{N-\mu})}
{\sqrt n+\sqrt{N-\mu}}>d_*.
\]
\end{proposition}
\begin{proof}
The equation follows by differentiation and $\mu p=\mu+2$.
Since $N/2<\mu<N$, the coefficient is smooth and has positive bounds.
With $t=N-\mu>0$,
\[
-\frac{rc'(r)}{c(r)}
=\frac{2(\mu+2)r^2}{(n+tr^2)(1+r^2)}.
\]
The maximum occurs at $r^2=\sqrt{n/t}$ and has the stated value.
Strictness follows from Proposition~\ref{prop:radialendpoint}.
\end{proof}

\subsection*{Acknowledgments}

\subsection*{Conflict of interest}

 No potential conflict of interest was reported by the authors.

\subsection*{Data Availability}

Data sharing is not applicable to this article as no new data were created or analyzed in this study.

\bibliographystyle{elsarticle-num}

\end{document}